\documentclass[10pt]{article}
\usepackage{amsfonts}
\usepackage{amssymb,latexsym}
\usepackage{amsmath,amscd,bbm}
\usepackage{amsthm}
\usepackage{cancel}
\usepackage[pdftitle={Paper}, pdfauthor={Alejandro Cabrera}]{hyperref}
\usepackage{color}
\usepackage[all, cmtip]{xy}
\usepackage[affil-it]{authblk}

\numberwithin{equation}{section}
\newtheorem{lemma} {Lemma} [section]
\newtheorem{proposition} [lemma] {Proposition}

\newtheorem{theorem} [lemma] {Theorem}
\newtheorem{corollary} [lemma] {Corollary}
\newtheorem{definition}[lemma] {Definition}

\newtheorem{example}[lemma] {Example}
\newtheorem{remark}[lemma]{Remark}
\renewenvironment{proof}{{\sc Proof:}}{{\hspace*{\fill} $\square$\\}}

\numberwithin{}{}

\newcommand{\thistheoremname}{}
\newtheorem{genericthm}[lemma]{\thistheoremname}

\newtheorem*{genericthm*}{\thistheoremname}
\newenvironment{namedthm*}[1]
  {\renewcommand{\thistheoremname}{#1}%
   \begin{genericthm*}}
  {\end{genericthm*}}
\makeatletter
\def\@maketitle{%
  \newpage
  \null
  \vskip 2em%
  \begin{center}%
  \let \footnote \thanks
    {\Large\bfseries \@title \par}%
    \vskip 1.5em%
    {\normalsize
      \lineskip .5em%
      \begin{tabular}[t]{c}%
        \@author
      \end{tabular}\par}%
    \vskip 1em%
    {\normalsize \@date}%
  \end{center}%
  \par
  \vskip 1.5em}
\makeatother

\def\vbgs{{VB-groupoids} \ }

\def\vbg{{VB-groupoid} \ }
\def\vba{{VB-algebroid} \ }

\def\bea{\begin{eqnarray*}}
\def\eea{\end{eqnarray*}}
\def\bl{\begin{lemma}}
\def\el{\end{lemma}}
\def\br{\begin{remark}}
\def\er{\end{remark}}

\def\G{\mathcal{G}}
\def\V{\V}

\def\l{\lambda}
\renewcommand{\hom} {\mathrm{hom}}
\newcommand{\khom} {k\mbox{-}\hom}
\def\C{{C^\infty}}
\def\Ck{{C^\infty_{\khom}}}

\def\CE{CE}

\def\VE{\mathrm{VE}}
\def\VEk{{\VE_{\khom}}}
\def\toto{{\rightrightarrows}}
\def\Ga{\Gamma}
\def\Gak{{\Gamma_{\khom}}}
\def\rra{\rightrightarrows}
\def\No{\mathbb{N}_0}
\def\ase{{\ast_E}}
\def\R {\mathbb{R}}
\def\Arrow {\rightarrow}
\def\i{\iota}
\def\<{\langle}
\def\>{\rangle}
\def\V {\mathcal{V}}

\def\D{\mathcal{D}}

\def\E{\mathcal{E}}

\newcommand{\sour}        {\mathsf{s}}
\newcommand{\tar}         {{\mathsf{t}}}
\newcommand{\mult}     {{\mathsf{m}}}
\renewcommand{\1}         {\mathbf{1}}
\newcommand{\g}         {\mathfrak{g}}
\newcommand{\h}         {\mathfrak{h}}
\renewcommand{\v}       {\mathfrak{v}}
\renewcommand{\H}       {\mathcal{H}}   
\newcommand{\bG}        {\mathbb{G}}
\newcommand{\bM}        {\mathbb{M}}
\newcommand{\bA}        {\mathbb{A}}

\newcommand{\bpartial}  {\cancel{\partial}}
\newcommand{\bs}        {\mathbbm{s}}
\newcommand{\frakf}         {\mathfrak{F}}

\newcommand{\vl}       {{\mathrm{v}}}       

\newcommand{\Lie}       {\mathcal{L}}
\newcommand{\pr}        {\mathrm{pr}}
\newcommand{\sk}        {\mathrm{sk}}
\newcommand{\spl}        {\mathrm{spl}}
\newcommand{\ext}       {\mathrm{ext}}
\newcommand{\f}         {\frakf}
\newcommand{\T}         {\mathbb{T}}
\newcommand{\Z}         {\mathbb{Z}}

\def\muE{\mu_{E}}
\def\muC{\mu_{C}}

\newcommand{\comment}{}

\begin{document}

\title{Van Est isomorphism for homogeneous cochains}
\author{Alejandro Cabrera%
  \thanks{Electronic address: \texttt{acabrera@labma.ufrj.br};}}
\affil{Departamento de Matematica Aplicada, Instituto de Matematica\\
Universidade Federal do Rio de Janeiro\\
 Caixa Postal 68530, Rio de Janeiro RJ 21941-909, Brasil.}

\author{Thiago Drummond%
  \thanks{Electronic address: \texttt{drummond@im.ufrj.br}}}
\affil{Departamento de Matematica, Instituto de Matematica\\
Universidade Federal do Rio de Janeiro\\
 Caixa Postal 68530, Rio de Janeiro RJ 21941-909, Brasil.}

\date{\today}

\maketitle

\begin{abstract}
VB-groupoids define a special class of Lie groupoids which carry a compatible linear structure. In this paper, we show that their differentiable cohomology admits a refinement by considering the complex of cochains which are $k$-homogeneous on the linear fiber. Our main result is a Van Est theorem for such cochains. We also work out two applications to the general theory of representations of Lie groupoids and algebroids. The case $k=1$ yields a Van Est map for representations up to homotopy on 2-term graded vector bundles. Arbitrary $k$-homogeneous cochains on suitable VB-groupoids lead to a novel Van Est theorem for differential forms on Lie groupoids with values in a representation.
\end{abstract}

\tableofcontents

\section{Introduction}
The Van Est theorem \cite{vE1, vE2} is a classical result relating the differentiable cohomology associated to a Lie group with the underlying Lie algebra cohomology. More precisely, given a Lie group $G$ with Lie algebra $\g$, the Van Est map is a chain map
$$\VE: C^p(G)=\{f\in \C(G^p): f(g_1,..,g_p)=0 \text{ if } g_i=e\} \longrightarrow CE(\g) = \Lambda^p \g^*$$ taking (normalized) differentiable $p$-cochains on $G$ to Lie algebra $p$-cochains. It is defined by 
\begin{equation}\label{1}
\VE(f)(u_1,..,u_p)= \sum_{\sigma \in S_p} sgn(\sigma) R_{u_{\sigma(1)}} \dots R_{u_{\sigma(p)}} f,
\end{equation}
where $R_u: C^p(G) \to C^{p-1}(G)$ is the operator which differentiates $f(\cdot, g_2, \dots, g_p)$ at the unit $e$ with respect to the right-invariant vector field corresponding to $u$. The map $\VE$ can be seen as a model for the pull-back of functions along the projection of the universal $G$-bundle $EG\to BG$. The Van Est theorem then states that, if $G$ is (topologically) $p_0$-connected, the map induced by $\VE$ in cohomology is an isomorphism for $p\leq p_0$ and injective for $p=p_0+1$.
In the setting of Lie groupoids, the Van Est Theorem was first studied by A. Weinstein and P. Xu \cite{XW} for $p_0=1$ and later generalized for arbitrary degrees by M. Crainic \cite{Cr} (see also the more recent work of D. Li-Bland and E. Meinrenken, \cite{L-BM}). 

In this paper, we provide a refinement of this theorem for a particular class of Lie groupoids endowed with a compatible linear structure, called VB-groupoids \cite{Prad2}(see also \cite{BCdH, GM08, GM10} and references therein). In this case, the linear structure allows us to refine the Van Est theorem by looking at \emph{homogeneous cochains} and  we are able to derive several interesting applications from this general result.

To illustrate our approach, we shall examine here a simple situation involving a Lie group $G$ and a linear representation $\rho_G: G \to Aut(V)$ on a (finite dimensional) real vector space $V$. The associated complex of differentiable cochains for $G$ with values in $V$ is $C^p(G,V) = \{ f: G^p \to V: f(g_1,..,g_p)=0 \ if \ g_i=e\}$ with a differential $\delta: C^p(G,V) \to C^{p+1}(G,V)$ which encodes $\rho_G$ (see Example \ref{ex:repG} below for an explicit formula). Infinitesimally, associated to the induced Lie algebra representation $\rho_\g: \g \to End(V)$, we have the Chevalley-Eilenberg complex of Lie algebra cochains with values in $V$, namely $CE^p(\g,V) = \Lambda^p \g^* \otimes V$. In this setting, there exists a natural analogue of the Van Est map
\begin{equation}\label{2}
\Psi_\rho: C^p(G,V) \to  \CE^p(\g,V).
\end{equation}
How to prove a Van Est theorem for $\Psi_\rho$? 
There are two approaches: the first one is to try re-prove the statement from scratch mimicking the proof of the standard case. 
The second one is to \emph{deduce} the desired result from the known Van Est theorem for {\em Lie groupoids} by relating the map \eqref{2} to the Van Est map $\VE$ for the action groupoid $\V=V^* \rtimes G$. It is the second approach we pursue in this paper.


To relate $\VE$ and $\Psi_\rho$, notice that both $\V$ and its space of $p$-composable arrows $B_p\V$ define vector bundles $\V \to G$ and $B_p\V \to G^p$, respectively. (Actually, $B_p\V$ is isomorphic to $V^* \times G^p$.) One can then show that the differentiable cochains $f \in C^{\infty}(B_p\V)$ which are fiberwise $k$-homogeneous define a subcomplex 
\[ 
C^p_{\khom}(\V) \subset C^p(\V).
\]
Analogously, the Lie algebroid $\v= V^* \rtimes \g$ of $\V$ also defines a trivial vector bundle $\v \to \g$ and the fiberwise $k$-homogeneous cochains define a subcomplex 
$$
CE^p_{\khom}(\v) \subset \Lambda^p \v^*.
$$
The key point is that $\VE$ preserves $k$-homogeneous cochains, thus restricting to a map
$$ \VE_{\khom}:  C^p_{\khom}(\V) \to CE^p_{\khom}(v)$$
which, from a simple homological algebra argument (see Section \ref{sec:vanest} below), is an isomorphism (resp. injective) in cohomology whenever $\VE$ is. Finally, to obtain the Van Est theorem for $V$-valued cochains one has to verify that
$$ H^p(C^\bullet_{1\mbox{-}\hom}(\V)) \simeq H^p(C^\bullet(G,V)), \ \ H^p(CE^\bullet_{1\mbox{-}\hom}(\v)) \simeq H^p(\Lambda^\bullet \g^* \otimes V), \ \ \VE_{1\mbox{-}\hom} \simeq \Psi_\rho.$$

In this paper, we follow the same reasoning but with $\V$ replaced by an arbitrary $VB$-groupoid. The main arguments follow directly as above but non-trivial computational effort needs to go into the last ingredient of the argument, namely, into relating the complexes of homogeneous (groupoid and algebroid) cochains to certain complexes already introduced in the literature from different perspectives. In particular, we shall obtain explicit formulae for the underlying Van Est maps. 

We work out two applications: in the first, we deduce a Van Est theorem for representations up to homotopy in 2-term graded vector bundles \cite{AC2, AC3,GM08,GM10} by looking at $1$-homogeneous cochains and generalizing the case of $\rho$ above, recovering results from \cite{AS}. In the case of the \emph{adjoint representation}, our approach can be seen as the realization of the original idea proposed in \cite{CrMo} for showing a rigidity result for certain proper groupoids. (This last result was also proven in \cite{AS} using different methods.) The second application provides a new Van Est theorem for differential forms on Lie groupoids with coefficients in a representation, generalizing the work of \cite{AC} on the Bott-Shulmann complex and \cite{CSS} on Spencer operators. It is interesting to notice that, in this second application, another idea is incorporated (which has its roots in \emph{supergeometry} and was used in a Lie-theoretic context by Mehta \cite{Mehta}): forms in $\Lambda^k V^*$ are $k$-homogeneous functions on $V^k$. 
For this application, we shall need the refinement of the Van Est theorem in its full extent (i.e. for $k$-homogeneous cochains, where $k$ is arbitrary). Even in the particular case of differential forms with trivial coefficients, our proof of the corresponding Van Est theorem is new and can be seen as illustration of the usefulness of homogeneous cochains.



\paragraph{Outline of the paper.} 
\begin{itemize}
\item In Section \ref{sec:homcoch} we set up some notation, introduce homogeneous cochains on VB-groupoids and algebroids and provide our main result: the corresponding refinement of the Van Est theorem.
\item In Section \ref{sec:ruth}, we speciallize to $1$-homogeneous cochains and deduce a Van Est result for representations up to homotopy. Along the way, we mention how this argument can be used to provide an alternative proof of the rigidity conjecture as originally proposed in \cite{CrMo}.
\item In Section \ref{sec:forms}, by means of $k$-homogeneous cochains in suitable VB-groupoids and algebroids, we prove a Van Est theorem for differential forms with coefficients in a representation.
\end{itemize}

To keep the main text as simple as possible, we decided to postpone to the Appendix some of the more technical or computational parts of the arguments in Section \ref{sec:forms}. Most of the explicit formulas contained in the Appendix follow from extensions of known lift properties of vector fields to Lie groupoids (see \cite{Mac-Xu, Mac-Xu3}). We would like to mention that part of this paper grew out of the project of understanding the Lie theory of multiplicative tensors on Lie groupoids \cite{bd}.

\bigskip
\noindent{\bf Acknowledgments.} We would like to thank H. Bursztyn and O. Brahic for useful discussions. A. C. would also like to thank R. Mehta for his insightful ideas in the early stages of this work.

\section{Homogeneous cochains and the Van Est map for VB-groupoids}\label{sec:homcoch}
In this section, we present a refinement of groupoid and algebroid cohomology theory for \vbgs and VB-algebroids by considering $k$-homogeneous cochains. We also show that an analogue of the Van Est theorem holds for such homogeneous cochains.

\subsection{Homogeneous functions on vector bundles}
Given any vector bundle $\pi: V\to B$, fiberwise multiplication by non-negative scalars $h: \R_+ \times V \Arrow V$ defines an action of the multiplicative monoid $\R_+$ which we shall call \textit{the homogeneous structure of $V \Arrow B$}. Following \cite{GR}, we recall that the homogeneous structure completely characterizes the underlying vector bundle structure and that, in particular, a smooth map between the total spaces defines a vector bundle morphism if and only if it commutes with the underlying $\R_+$-actions. (See \cite{BCdH} for applications of these ideas in a Lie theoretic context.)

For each $k\in \mathbb{N}$, we consider
$$
\Ck(V) := \{f\in \C(V): h_\l^*f = \l^k f, \ \forall \l \geq 0\},
$$
the set of \text{fiberwise $k$-homogeneous functions}. Note that
$$
C^{\infty}_{0\mbox{-}\hom}(V) = \{f \in \C(V): \exists f_0 \in \C(B), \,  f = f_0\circ \pi\} \cong C^\infty(B).
$$
Multiplication of functions gives a map
$
C_{\khom}^{\infty}(V) \times C_{k'\mbox{-}\hom}^\infty(V) \to C_{k+k'\mbox{-}\hom}^\infty(V)
$
and, in particular, each $C^{\infty}_{\khom}(V)$ is a $C^{\infty}(B)$-module. In fact,
$
C^{\infty}_{\khom}(V) \cong \Gamma(B,S^k V^*),
$
for the symmetric algebra bundle $S^{k} V^*\to B$. The isomorphism $\Ga(B,V^*) \cong C^{\infty}_{1\mbox{-}\hom}(V)$ takes a section $\mu \in \Gamma(B, V^*)$ to the fiberwise-linear function $\ell_{\mu} \in C^{\infty}_{1\mbox{-}\hom}(V)$ given  by
$$
\ell_{\mu}(v) = \<\mu(b), v\>, \,\,\, v \in V_b, \, b \in B.
$$
The $k$-th derivative along the fiber defines a projection $P_{\khom}: C^{\infty}(V) \to \Ck(V)$, 
\begin{equation}\label{proj:hom}
P_{\khom}(f) = \frac{1}{k!}\frac{d^k}{d\l^k}(h_\l^*f)|_{\l=0}.
\end{equation}
If $(x, \xi_1, \dots, \xi_n)$ are trivializing coordinates on $V$, then
$$
P_{\khom} (f) (x,\xi) = \sum_{k_1 +\dots+ k_n=k} \frac{1}{k_1!\dots k_n!}\frac{\partial^k f}{\partial \xi_1^{k_1}\dots \partial \xi_{n}^{k_n}} (x,0) \,\xi_1^{k_1} \dots \xi_n^{k_n}.
$$

\subsection{Homogeneous groupoid cochains}
Let $\G \rightrightarrows M$ be a Lie groupoid with source and target maps, $\sour, \tar: \G \Arrow M$, unit $\mathbf{1}: M \to \G$, inversion $\iota: \G \to \G$ and multiplication $\mult: \G_{\,\,\sour}\!\!\times_\tar \G \to \G$. We denote by $B_p\G$ the manifold of composable $p$-tuples ($B_0\G=M)$. The nerve of $\G$ is the simplicial manifold whose space of $p$-simplices is $B_p\G$ with the simplicial structure given by the face maps: $\partial_i : B_p\G \to B_{p-1}\G$, $i=0,\dots, p$, defined by
$$
\partial_i(g_1, \dots, g_p) =
\begin{cases}
(g_2, \dots, g_p), & \text{ if } i=0,\\
(g_1, \dots, g_{i-1}, g_{i}g_{i+1},  g_{i+2}, \dots, g_{p}), & \text{ if } 1 \leq i \leq p-1,\\
(g_1, \dots, g_{p-1}), & \text{ if } i=p,
\end{cases}
$$
and the degeneracy maps: $s_i: B_{p-1}\G \to B_p\G$, $i=0,\dots, p-1$, defined by
$$
s_i(g_1,\dots, g_{p-1}) = (g_1, \dots, g_{i}, 1_{\tar(g_{i+1})}, g_{i+1}, \dots, g_{p-1}).
$$
For $p=1$, $\partial_0 = \tar, \, \partial_1 = \sour$ and $s_0 = \mathbf{1}$.

The nerve defines a functor $B_{\bullet}$ from the category of Lie groupoids to the category of simplicial manifolds. For a groupoid morphism $\phi: \G_1 \to \G_2$, the morphism $B\phi: B\G_1 \to B\G_2$ is defined by $B_p\phi(g_1, \dots, g_p) = (\phi(g_1), \dots, \phi(g_p))$.

The space of \textit{(normalized) p-cochains} $C^p(\G)$ on $\G$ consists of smooth functions $f: B_p\G \to \R$ such that $s_i^* f=0$, for $i=0,\dots, p-1$. These define a cochain complex with differential
$
\delta : C^{p-1}(\G)  \to C^{p}(\G) 
$
defined by
\begin{equation}\label{G_diff}
\delta = \sum_{i=0}^{p} (-1)^{i} \partial_i^*.
\end{equation}
The \emph{differentiable cohomology} of $\G$ is the cohomology of the complex $(C^{\bullet}(\G), \delta)$ and we shall denote it by $H^{\bullet}(\G)$.
For $f_1 \in C^p(\G), f_2 \in C^{p'}(\G)$, the cup product $f_1 \star f_2 \in C^{p+p'}(\G)$ is defined by
\begin{equation}\label{cup_prod}
(f_1 \star f_2)(g_1, \dots, g_{p+p'})= f_1(g_1, \dots, g_p)f_2(g_{p+1}, \dots, g_{p+p'}).
\end{equation}
It defines an algebra structure on $C^{\bullet}(\G)$ which passes to cohomology due to the Leibniz formula
$$
\delta(f_1 \star f_2) = \delta(f_1)\star f_2 + (-1)^p f_1 \star \delta(f_2).
$$

In the following, we shall investigate how the differentiable cohomology of a \emph{VB-groupoid} interacts with its underlying homogeneous structure. 

\begin{definition}
A VB-groupoid is given by a commutative square
\begin{equation}\label{vb_diag1}
\xymatrix{
\V \ar@<-3pt>[d] \ar@<3pt>[d] \ar[r] & \G \ar@<-3pt>[d]\ar@<3pt>[d]\\
E \ar[r] & M,\\
}\end{equation}
where the left and right sides are Lie groupoids and the top and bottom sides are vector bundles satisfying the following compatibility condition: 
\begin{equation}
\label{vb_diag2}
\xymatrix{
\V \ar@<-3pt>[d] \ar@<3pt>[d] \ar[r]^{h_\l^{\scriptscriptstyle \G}} & \V \ar@<-3pt>[d]\ar@<3pt>[d]\\
E \ar[r]_{h_\l} & E,\\
}\end{equation}
defines a Lie grupoid morphism for each $\lambda \in \R_+$, where $h_\l^\G: \V \to \V$, $h_\l: E \to E$ are the homogeneous structures corresponding to $\V \to \G$, $E \to M$, respectively. We shall denote the structure maps of $\V\rightrightarrows E$ by $\sour_\V,\tar_\V, \mathbf{1}_\V, \iota_\V, m_\V$.
\end{definition} 

Instead of looking at the homogeneous structure, VB-groupoids can be alternatively defined by focusing on the fiberwise defined sum (see \cite{GM10}). Our choice of definition comes from \cite{BCdH} where the two definitions are shown to be equivalent (see Theorem 3.2.3 therein). 

VB-groupoids have found several applications in the latter years (\cite{BC,BCdH,bd,GM10,Mac-Xu, Mac-Xu3} just to mention a few). Natural examples of VB-groupoids are given by the tangent $T\G\rightrightarrows TM$ and cotangent $T^*\G\rightrightarrows Lie(\G)^*$ groupoids, which provide intrinsic versions of the adjoint and coadjoint representations (up to homotopy, see Section \ref{sec:ruth} below)  of a Lie groupoid $\G$. Ordinary representations also provide examples of VB-groupoids, as we shall see in detail in Example \ref{ex:repG} below. 

From now on, we shall focus on introducing \emph{homogeneous cochains} on a VB-groupoid and to study their properties with respect to the Van Est map, while having in mind the applications to be developed in Sections \ref{sec:ruth} and \ref{sec:forms}.
The first result states that $B_\bullet$ restricts to a functor from VB-groupoids to simplicial vector bundles.

\begin{lemma}\label{p_vector}
Let $\V \toto E$ be a VB-groupoid over $\G \toto M$. The space of $p$-composable arrows $B_p\V$ is a vector bundle over $B_p\G$. Moreover, the face and degeneracy maps are all vector bundle maps.
\end{lemma}

\begin{proof}
Consider $\V^p= \V \times \dots \times \V$ as a vector bundle over $\G^p$. We shall present $B_p\V$ as a subbundle of $\V^p$ restricted to $B_p\G \subset \G^p$. It follows from the commutativity of \eqref{vb_diag1} that  $B_p\V$ projects onto $B_p\G$. As $B_p\V$ is a smooth submanifold of $\V^p$, it remains to check that it is invariant by the homogeneous structure of the vector bundle $\V^p \to \G^p$ (see \cite{GR}). This is a straightforward consequence of the fact that \eqref{vb_diag2} is a groupoid morphism. The statement regarding the face and degeneracy maps follows now from the fact that the multiplication $m_\V: B_2\V \to \V$ is a vector bundle map (see also \cite{BCdH}).
\end{proof}

Note that the homogeneous structure $h_\l^{B_p\G}: B_p\V \to B_p\V$ of the vector bundle $B_p\V \to B_p\G$ satisfies
$$
B_p h_\l^\G = h_\l^{B_p\G}.
$$

It is now a straightforward consequence of Lemma \ref{p_vector} that homogeneous cochains define a subcomplex of the differentiable cohomology of $\V$.

\begin{proposition}
Let $\V \rra E$ be a VB-groupoid. If $P_{\khom}^{\G,p}: \C(B_p\V) \to \Ck(B_p\V)$ is the projection \eqref{proj:hom} induced by $h^{B_p\G}_\l$, then
$$
P_{\khom}^{\G,p+1} \circ \delta = \delta \circ P^{\G,p}_{\khom}.
$$
In particular, 
$$
\delta ( \Ck(B_p\V) ) \subset \Ck(B_{p+1}\V).
$$
\end{proposition}

Thus, for a \vbg $\V \toto E$, we define natural subcomplexes of $(C^{\bullet}(\V), \delta)$ by considering the set of fiberwise $k$-homogeneous functions,
$$
C^{\bullet}_{\khom}(\V): =  \Ck(\V^{(\bullet)}) \text{ and }  H^{\bullet}_{\khom}(\V) = H(C^{\bullet}_{\khom}(\V)).
$$ 

\begin{remark}\em
For $k=0$, $C^{\bullet}_{0\mbox{-}\hom} (\V) \simeq C^\bullet(\G)$ and the cup product \eqref{cup_prod} on $C^\bullet(\V)$ induces a right $C^{\bullet}(\G)$-module (resp. $H^\bullet(\G)$-module) structure on $C^{\bullet}_{\khom}(\V)$ (resp. $H^\bullet_{\khom}(\V)$).
\end{remark}

\begin{example}\label{ex:repG}
Let $C \to M$ be a (left) representation of the Lie groupoid $\G \toto  M$. The vector bundle $\V=\tar^*C^* \to \G$ carries a VB-groupoid structure $\tar^* C^* \toto C^*$ defined by
\begin{align*}
& \sour_\V(g, \xi) = \Delta_g^*(\xi) , \,\,\, \, \tar_\V(g,\xi) = \xi\\ 
& \iota_\V(g,\xi) = (g^{-1}, \Delta_g^*(\xi)), \,\, \mathbf{1}_\V(\xi)= (\mathbf{1}_{\pi(\xi)}, \xi) \,\text{ and } \, \mult_\V((g,\xi_1), (h,\xi_2)) = (gh, \xi_1),
\end{align*}
where $\Delta_g: C_{\sour(g)} \to C_{\tar(g)}$ is the action of $g \in \G$. Note that $\tar^*C^* = C^* \rtimes \G$, the action groupoid for the adjoint action of $\G$ on $C^*$. As vector bundles over $B_p\G$, one has that $B_p(\tar^* C^*) = \tar_p^*\,C^*$
where $\tar_p: B_p\G \to M$ is given by $\tar_p(g_1, \dots, g_p) = \tar(g_1)$ and the isomorphism is given by
$
((g_1, \xi_1), \dots, (g_p, \xi_p)) \mapsto ((g_1, \dots, g_p), \xi_1).
$
In particular,
$$
C_{1\mbox{-}\hom}^p(\V) \cong \Gamma(B_p\G, \tar_p^*\,C).
$$
The right $C^\bullet(\G)$-module structure on $C_{1\mbox{-}\hom}^\bullet(\V)$ corresponds to a right module structure on $\Gamma(B_\bullet\G, \tar^*_\bullet C)$ given by
\begin{equation}\label{star_prod}
(\phi \star f)(g_1, \dots, g_{p+p'}) =  \phi(g_1,\dots, g_p) \, f(g_{p+1}, \dots, g_{p+p'}), \, f \in C^{p'}(\G), \,\phi \in \Gamma(B_p\G, \tar_p^*C).
\end{equation}
Moreover, the differential on $C_{1\mbox{-}\hom}^{\bullet}(\V)$ corresponds to the differential on $\Gamma(B_{\bullet}\G, \tar_\bullet^*\,C)$ given by
$$
(\delta \phi)(g_1, \dots, g_{p+1}) = \Delta_{g_1}(\phi(g_2, \dots, g_p)) + \sum_{i=1}^{p-1} (-1)^i \phi(g_1, \dots, g_ig_{i+1}, \dots, g_p) + (-1)^p \phi(g_1,\dots, g_{p-1}).
$$
Hence, as $H^\bullet(\G)$-modules, $H^{\bullet}_{1\mbox{-}\hom}(\V) \cong H^\bullet(\G, C)$,
the cohomology of $\G$ with coefficients on the representation $C$ (see \cite{Cr}). 
More generally, $H^{\bullet}_{\khom}(\V) \cong H^{\bullet}(\G, S^kC)$.
\end{example}

\subsection{Homogeneous algebroid cochains}
Given a VB-groupoid $\V \toto E$, its Lie algebroid $\v \to E$ inherits the structure of a VB-algebroid (see \cite{BCdH} and references therein). As for VB-groupoids, we take our working definition from \cite{BCdH}.

\begin{definition}
A VB-algebroid is given by a commutative square
\begin{equation}\label{vba_diag1}
\xymatrix{
\v \ar[d] \ar[r] & \g \ar[d]\\
E \ar[r] & M,\\
}\end{equation}
where the left and right sides are Lie algebroids and the top and bottom sides are vector bundles satisfying the following compatibility condition: 
\begin{equation}
\label{vba_diag2}
\xymatrix{
\v \ar[d] \ar[r]^{h_\l^{\scriptscriptstyle \g}} & \v \ar[d]\\
E \ar[r]_{h_\l} & E,\\
}\end{equation}
defines a Lie algebroid morphism for each $\lambda \in \R_+$, where $h_\l^\g$, $h_\l$ are the homogeneous structures of the vector bundles $\v \to \g$ and $E \to M$, respectively. 
\end{definition}

Parallel to VB-groupoids, VB-algebroids together with Lie theory for VB-objects have found several applications in the latter years (again, we list a just few of the available references \cite{BC,BCdH,bd,GM08,  Mac-Xu, Mac-Xu3}). The tangent $TA\to TM$ and the cotangent lift   $T^*A \to A^*$ define examples of VB-algebroids corresponding to $T\G$ and $T^*\G$ when $A=Lie(\G)$,  providing intrinsic versions of the adjoint and coadjoint representations (up to homotopy, see Section \ref{sec:ruth} below)  of a Lie algebroid $A$. Ordinary representations of $A$ also provide examples of VB-groupoids, as explained in Example \ref{example:rep} below. We shall now investigate the infinitesimal version of the notion of \emph{homogeneous cochains}.

For any Lie algebroid $A \to M$, let $\CE^p(A):=\Ga(M,\Lambda^p A^*)$ and
$
d: \CE^p(A)\to \CE^{p+1}(A)
$
be the (Chevalley-Eilenberg) differential. The Lie algebroid cohomology $H^{\bullet}(A)$ is the cohomology of the complex $(\CE^{\bullet}(A), d)$. The wedge product on $\Gamma(M, \Lambda^\bullet A^*)$ induces a graded commutative algebra structure on $H^\bullet(A)$. 

When considering a VB-algebroid $A=\v$, the dual $\v^*$ is always taken with respect to the Lie algebroid side $\v \to E$, so that $\CE^p(\v) =  \Ga(E,\Lambda^p \v^*)$. The space of fiberwise (with respect to $\v \to \g$) $k$-homogeneous $p$-forms on $\v \to E$ is
\begin{equation}
\Gak(E,\Lambda^p \v^*) := \{\alpha \in \Ga(E,\Lambda^q \v^*):  h_{\l}^{\g \,\,*} \alpha = \l^k \alpha, \ \forall \l \geq 0\}.
\end{equation}

The wedge product induces a map
$$
\cdot \wedge \cdot: \Gamma_{\khom}(E, \Lambda^p \v^*) \times \Gamma_{k'\mbox{-}\hom}(E, \Lambda^{p'} \v^*) \to \Gamma_{k+k'\mbox{-}\hom}(E,  \Lambda^{p+p'} \v^*).
$$

Similarly to \eqref{proj:hom}, there exists a projection $P_{\khom}^{\g,p}: \Ga(E, \Lambda^p \v^*) \to \Gak(E, \Lambda^p \v^*)$ defined by
\begin{equation}\label{proj:hom2}
P_{\khom}^{\g,p}\alpha = \frac{1}{k!}\frac{d^k}{d\l^k}(h_\l^{\g\,\,*}\alpha)|_{\l=0}.
\end{equation}

\begin{proposition}
Let $\v \to E$ be a VB-algebroid. For each $k\in \No$ and every $p \geq 0$,
$$
P_{\khom}^{\g,p+1} \circ d = d \circ P_{\khom}^{\g,p}.
$$
In particular,
$$ 
d(\Gak(E,\Lambda^p \v^*)) \subset \Gak(E,\Lambda^{p+1} \v^*)
$$
\end{proposition}

\begin{proof}
Since the Chevalley-Eilenberg differential $d$ is a local operator, we can assume $\v \to E$ is trivial. By looking at $h^{\v \,\,*}_\lambda \alpha$ as a smooth 1-parameter family of forms, one can see that $d$ commutes with $d/d\lambda$. The statement then follows from the fact that $h_\l^\v$ is a Lie algebroid morphism and, hence, $h_\l^{\v\,\,*}$ commutes with $d$.
\end{proof}

Thus, for each $k \in \No$, the $k$-homogeneous forms define a subcomplex $\CE^{\bullet}_{\khom}(\v)$ of $(\CE^\bullet(\v),d)$. The notation we shall use is
$$
\CE^p_{\khom}(\v) := \Gak(E,\Lambda^p \v^*) \text{ and } H_{\khom}^\bullet(\v) = H(\CE^{\bullet}_{\khom}(\v)).
$$

\br 
For $k=0$, $
\Gamma_{0\mbox{-}\hom}(E, \Lambda^p \v^*) \cong \Gamma(M,\Lambda^p\g^*)$
and the wedge product turns $\Gamma_{\khom}(E, \Lambda^\bullet \v^*)$ (resp. $H_{\khom}^{\bullet}(\v)$) into a right $\Gamma(M,\Lambda^\bullet \g^*)$-module (resp. $H^\bullet(\g)$-module).
\er

\begin{example}\label{example:rep}
Let $C \to M$ be a representation of the Lie algebroid $\g \to M$ defined by a $\g$-connection $\nabla: \Gamma(\g)\times \Gamma(C) \to \Gamma(C)$. Consider the vector bundle $\v = C^* \times_M \g \to C^*$. Given $u \in \Gamma(\g)$, let $\chi_u: C^* \to \v$ be the section given by
\begin{equation}\label{linear_C}
\chi_u(\xi) = (\xi, u(m)), \, \text{ for } \xi \in C_m^*.
\end{equation}
The sections $\chi_u$ with $u$ varying on $\Gamma(\g)$ generate $\Gamma(C^*, \v)$ as a $C^{\infty}(C^*)$-module. One can now show that the action algebroid structure $C^* \rtimes \g \to C^*$, determined by
\begin{align*}
[\chi_{u_1}, \chi_{u_2}] & = \chi_{[u_1,u_2]},\,\, \,u_1, \,u_2 \in \Gamma(\g);\\
\rho_\v(\chi_{u_1})(\ell_{\mu}) & = \ell_{\nabla_{u_1} \mu} \text{ and } \rho_\v(\chi_{u_1})(f \circ \pi) = (\Lie_{\rho(u_1)} f) \circ \pi, \,\, f \in C^\infty(M), \, \mu \in \Gamma(C),
\end{align*}
endows $\v \to C^*$ with a VB-algebroid structure, where $\pi: C^* \to M$ is the projection. The chain complex $\CE^{\bullet}_{1\mbox{-}\hom}(\v)$ is naturally isomorphic to $\Gamma(\Lambda^{\bullet} \g^* \otimes C)$ with the Koszul differential
\begin{align*}
d_{\nabla}\gamma(u_1, \dots, u_{p+1}) & = \sum_{i=1}^{p+1} (-1)^{i+1}\nabla_{u_i} \gamma(u_1, \dots, \hat{u_i}, \dots, u_{p+1})\\
& \hspace{-50pt} + \sum_{1 \leq i < j \leq p+1} (-1)^{i+j} \gamma([u_i, u_j], u_1, \dots, \hat{u_i}, \dots, \hat{u_j}, \dots, u_{p+1}), \, \gamma \in \Gamma(\Lambda^p \g^*\otimes C).
\end{align*}
More precisely, the evaluation map $ev: \CE^{p}_{1\mbox{-}\hom}(\v) \to \Gamma(\Lambda^p \g^* \otimes C)$,
$$
\<ev(\alpha)(u_1, \dots, u_p), \xi\> = \alpha(\chi_{u_1}(\xi), \dots, \chi_{u_p}(\xi)), \,\, \text{for} \,\,\, u_1, \dots, u_p \in \Gamma(\g), \, \xi \in C^*
$$
defines a chain isomorphism. The induced right $\Gamma(\Lambda^\bullet \g^*)$-module structure on $\Gamma(\Lambda^\bullet \g^* \otimes C)$ is wedge multiplication on the right in the $\Lambda \g^*$ factor. In particular, as $H(\g)$-modules, $H_{1\mbox{-}\hom}^{\bullet}(\v) \cong H^{\bullet}(\g, C)$, the cohomology of $\g$ with values in the representation $C$. As for groupoids, $H_{\khom}^{\bullet}(\v) \cong H^{\bullet}(\g, S^kC)$.

\end{example}

\subsection{Van Est theorem for homogeneous cochains}\label{sec:vanest}
Let $\G \rra M$ be a Lie groupoid with Lie algebroid $\g$. For every section $u \in \Gamma(\g)$, consider the corresponding right invariant vector field $\overrightarrow{u} \in \mathfrak{X}(\G)$. In the following, we shall denote by $B_pu$ the vector field on the space of $p$-composable arrows $B_p\G$ given by
\begin{equation}\label{Bp_vector}
B_pu(g_1,\dots, g_p) = (\overrightarrow{u}(g_1), 0_{g_2}, \dots, 0_{g_p}).
\end{equation}
Let us now recall the definition of the Van Est map. First, using the degeneracy map $s_0: B_{p-1}\G \to B_p\G$, we define $R_{u}: C^p(\G) \to C^{p-1}(\G)$ by
$$
R_{u} = s_0^* \circ \Lie_{B_pu}.
$$ 
The Van Est map
$\VE:  C^p(\G) \to \CE^p(\g)$
is defined as follows (\cite{Cr}): for a $p$-cochain $f \in C^p(\G)$, 
\begin{equation}\label{vanest}
\VE(f)(u_1, \dots, u_p) = \sum_{\sigma \in S_p} sgn(\sigma) \, R_{u_{\sigma(1)}}\dots R_{u_{\sigma(p)}} (f).
\end{equation}
In \cite{Cr} it is shown that it defines a chain map between the underlying complexes which preserves the corresponding product structures. We shall also need the following naturality result about $\VE$.

\begin{lemma}\label{commut_vannest}
Let $\mathcal{H}_1, \mathcal{H}_2$ be Lie groupoids with Lie algebroids $\mathfrak{h}_1, \mathfrak{h}_2$, respectively. If $\phi: \mathcal{H}_1 \to \mathcal{H}_2$ is a Lie groupoid morphism with the corresponding Lie algebroid morphism $Lie(\phi): \mathfrak{h}_1 \to \mathfrak{h}_2$, then
$$
\VE(B_p\phi^*f) = Lie(\phi)^* \VE(f), \ \ \forall f \in C^p(\mathcal{H}_2).
$$
\end{lemma}

\begin{proof}
For any $\chi\in \Gamma (\h_1)$ we can write
$$ Lie(\phi)(\chi) = \sum_i \gamma_i \ (\tilde{\chi}_i\circ \phi_0) \in \Gamma (\phi_0^*\h_2),$$
where $\phi_0=B_0\phi: M_1 \to M_2$ denotes the map between objects induced by $\phi$, $\gamma_i \in \C(M_1)$ and $\tilde{\chi}_i \in \Gamma(\h_2)$. A direct computation shows that 
$$R_{\chi}((B_p\phi)^*f) =\sum_i (\tar_{p-1}^*\gamma_i) (B_{p-1}\phi)^* (R_{\tilde{\chi}_i}f), \,\,\,   \forall f\in C^p(\mathcal{H}_2).$$
If we apply the above formula $p$-times, we notice that most terms in $R_{\chi_1}...R_{\chi_p} (B_p\phi)^*f$ will vanish since $\VE$ is defined on normalized cochains (namely, $s_i^*f=0$) but the term
$$ \sum_{i_1,..,i_p} \gamma_{i_1} ... \gamma_{i_p} \phi_0^*(R_{\tilde{\chi}_1}...R_{\tilde{\chi}_p}f).$$
 We thus get  the statement of the Lemma.
\end{proof}

The main result about the Van Est map in the present context is the following Theorem due to M. Crainic.

\begin{theorem}[\cite{Cr}] \label{vanest_thm}
Let $\G$ be a Lie groupoid and let $\g$ be its Lie algebroid. The Van Est map \eqref{vanest} induces an algebra homomorphism 
$$
VE: H^{\bullet}(\G) \to H^{\bullet}(\g).
$$
Moreover, if $\G$ has $p_0$-connected source fibers, then $VE$  is an isomorphism in degrees $p \leq p_0$, and it is injective for $p=p_0+1$.
\end{theorem}

To get our refinement of Theorem \ref{vanest_thm} for homogeneous cochains on VB-groupoids and algebroids, we first state the following simple homological algebra fact. 

\begin{namedthm*}{Homological Lemma}
Let $(C^{\bullet}_i, \delta_i)$ be differential complexes, $i=1,2$, endowed with projections $P_i: C^\bullet_i \to C^\bullet_i$ (i.e. $P_i \circ \delta_i = \delta_i \circ P_i$ and $P_i^2=P_i$). If $F: C_1^\bullet \to C_2^\bullet$ is a morphism satisfying $F \circ P_1 = P_2 \circ F$, then for each $p$ such that $F: H^{p}(C_1) \to H^p(C_2)$ is injective (resp. surjective) its restriction $F_r: H^p(S_1) \to H^p(S_2)$ is also injective (resp. surjective), where $S_i^\bullet = P_i(C_i^\bullet)$.
\end{namedthm*}

We are thus left with studying the behaviour of the projections onto homogeneous cochains under the Van Est map.
%
%
To that end, let $\V \toto E$ be a VB-groupoid over $\G \toto M$ and let $\v \to E$ be its Lie algebroid. 

\begin{proposition}\label{hom_vanest}
For each $k \in \No$ and every $p \geq 0$,
$$
\VE \circ P_{\khom}^{\G,p} = P_{\khom}^{\g,p} \circ \VE.
$$
In particular, $\VE(\Ck(B_p\V)) \subset \Gak(\Lambda^p \,\v^*)$.
\end{proposition}

\begin{proof}
Let $h_{\l}^{\G}: \V \to \V$, $h_{\l}^{\g}: \v \to \v$ be the homogeneous structures of the vector bundles $\V \to \G$, $\v \to \g$, respectively. By Lemma \ref{commut_vannest}, the fact that $h_{\l}^{\G}$ is a groupoid homomorphism with $Lie(h_\l^\G) = h_\l^\g$ implies that
$$
\VE \circ h_\l^{\G \,\, *} = h_\l^{\g\,\,*} \circ \VE, \ \forall \lambda.
$$
Hence, by applying $\frac{d}{d\l}|_{\l=0}$ on both sides, one obtains the commutation relation between $\VE$ and the projections $P_{hom,k}^{\cdot\,,\, p}$. The result now follows directly.
\end{proof}

The restriction of the Van Est map to the subcomplex of $k$-homogeneous cochains shall be denoted by
$$
\VEk : = \VE|_{C_{\khom}^p(\V)} : C^{p}_{\khom}(\V) \to CE^p_{\khom}(\v).
$$

\begin{example}[0-homogeneous cochains]
For $k=0$, using the isomorphisms $C_{0\mbox{-}\hom}^p(\V) \cong C^p(\G)$ and $CE^p_{0\mbox{-}\hom}(\v) \cong \CE^p(\g)$, one can check that $\VE_{0\mbox{-}\hom} \cong \VE_\G: C^p(\G) \to \CE^p(\g)$. To see this, take $f \in C_{0\mbox{-}\hom}^p(\V)$ and $\chi_1, \dots, \chi_p \in \v$, and notice that $\VE_{0\mbox{-}\hom}(f)(\chi_1, \dots, \chi_p)$ only depends on the projections $u_i \in \g$ of $\chi_i$, $i=1,\dots, p$. Hence, to compute $\VE_{0\mbox{-}\hom}$, it suffices to take $\chi_1, \dots, \chi_p$ linear sections \footnote{A linear section $\chi$ of $\v$ is a section $\chi: E \to \v$ which is a vector bundle homomorphism covering a section $u: M \to \g$ (see \cite{GM08}).} of $\v$ covering $u_1, \dots, u_p \in \Gamma(\g)$. In this case, 
$$
\VE_{0\mbox{-}\hom}(f)(\chi_1, \dots, \chi_p) = \pi_E^*\VE_\G(f_0)(u_1, \dots, u_p),
$$ 
where $f= \pi_{B_p\G}^*f_0$, $f_0 \in C^{\infty}(B_p\G)$ and  $\pi_E: E \to M$, $\pi_{B_p\V}: B_p\V \to B_p\G$ are the vector bundle projections.
\end{example}

We are now ready to state and prove our main theorem. 
\begin{theorem}\label{main}
Let $\G \toto M$ be a Lie groupoid with Lie algebroid $\g$. For a \vbg $\V \toto E$ over $\G$ with underlying \vba $\v\to E$, the Van est map on $k$-homogeneous cochains induces a module homomorphism
$$
\VEk: H^\bullet_{\khom}(\V) \to H^\bullet_{\khom}(\v)
$$
covering the algebra homomorphism $\VE_\G: H^{\bullet}(\G) \to H^{\bullet}(\g)$. Moreover, if $\G$ has $p_0$-connected source fibers, then $\VE_{\khom}$ is an isomorphism for all $p\leq p_0$ and it is injective for $p=p_0+1$.
\end{theorem}

\begin{proof}
The $H^\bullet(\G)$-module structure on $H_{\khom}^\bullet(\v)$ comes from the cup product of $C^\bullet_{\khom}(\V)$ and $C_{0\mbox{-}\hom}^\bullet(\V) \cong C^\bullet(\G)$. So, the first statement follows from the fact that $\VEk$ is the restriction of the Van Est map of $\V$ to homogeneous cochains and that $\VE_{0\mbox{-}\hom} \cong \VE_\G$. 

Let us now assume that $\G$ has $p_0$-connected source fibers. First note that this implies that $\V \toto E$ is also source $p_0$-connected. Indeed, a source fiber of $\V \toto E$ is a affine bundle over the corresponding source fiber of $\G \toto M$. So, the Van Est Theorem \ref{vanest_thm} implies that $\VE: H^p(\V) \to H^p(\v)$ is an isomorphism for $p \leq p_0$ and injective for $p=p_0+1$. The result will now follow from Proposition \ref{hom_vanest} by applying the Homological Lemma to $F= \VE$, $(C_1^\bullet, \delta_1) = (\C(B_{\bullet}\V), \delta)$,$(C_2^\bullet, \delta_2) = (\Gamma(E,\wedge^\bullet \v^*), d)$ with projections  $P_1 = P_{\khom}^{\G,\bullet}:\C(B_{\bullet}\V) \to \Ck(B_{\bullet}\V)$ and $P_2=P_{\khom}^{\g, \bullet}: \Gamma(E,\Lambda^\bullet \v^*) \to \Gak(E,\Lambda^\bullet \v^*)$. 
\end{proof}

\section{1-Homogeneous cochains and representations up to homotopy}\label{sec:ruth}
In \cite{GM08,GM10}, it was shown that VB-groupoids and VB-algebroids provide an intrinsic version of the notion of (2-term) representation up to homotopy, generalizing the example given in the introduction, as well as Examples \ref{ex:repG} and \ref{example:rep} above. In this section, we show how Theorem \ref{main}, when applied to $1$-homogeneous cochains, recovers a Van Est result for the underlying 2-term representations up to homotopy of \cite{AS}. We also comment on how this approach realizes the original porposal in \cite{CrMo} for proving a rigidity conjecture.

\subsection{VB-groupoid and VB-algebroid cohomology}\label{sec:vbcom}
Following \cite{GM10}, given \vbg $\pi:\V \to \G$ we define 
$C^p_{VB}(\V)$ to be the space 
of $1$-homogeneous cochains $\phi \in C^\infty_{1\mbox{-}\hom}(B_p\V)$ satisfying the following two additional conditions: 
\begin{enumerate}
\item \label{cond1} $\phi(0_g, \xi_1,..,\xi_{p-1}) = 0$,
\item \label{cond2} $\phi(0_g\cdot \xi_1,..,\xi_{p}) = \phi(\xi_1,..,\xi_{p})$, 
\end{enumerate}
for all $(\xi_1,..,\xi_p) \in B_p\V$ and $g\in \G$ such that $(0_g,\xi_1)\in B_2\V$.
As observed in \cite{GM10}, condition \ref{cond1} above implies that $\phi(\xi_1,\xi_2,..,\xi_p)$ only depends on $\xi_1$ and on the projections $g_i=\pi(\xi_i)\in \G, i=1,..,p$, while condition \ref{cond2} is a left-invariance property. 

It is shown in \cite{GM10} that $C^\bullet_{VB}(\V)$ defines a subcomplex of $C^\bullet_{1\mbox{-}\hom}(\V)$. Moreover, the cup product with $C_{0\mbox{-}\hom}^\bullet{\V} \cong C^\bullet(\G)$ defines a right $C^\bullet(\G)$-submodule structure on $C^\bullet_{VB}(\V)$. The following Lemma relates the cohomology of the two complexes.

\bl\label{lma:vbgc0}
The inclusion $\i: C_{VB}^{\bullet}(\V) \hookrightarrow C_{1\mbox{-}\hom}^{\bullet}(\V)$ induces a isomorphism of right $H^{\bullet}(\G)$-modules in cohomology.
\el

\begin{proof}
We need to show that every cocycle in $ C^\infty_{1\mbox{-}\hom}(B_p\V)$ is cohomologous to an element of the subcomplex $C^p_{VB}(\V)$. To that end, first notice that if $\delta \phi =0$ and $\phi$ satisfies condition \ref{cond1}, then it satisfies condition \ref{cond2}. This follows directly from evaluating \[ 0= (\delta \phi)(0_g,\xi_1,...,\xi_p).\]
We are thus left with showing that,  for each cocycle $\phi \in C^\infty_{1\mbox{-}\hom}(B_p\V)$ there exists a $\psi \in C^\infty_{1\mbox{-}\hom}(B_{p-1}\V)$ such that $\phi + \delta \psi$ satisfies \ref{cond1}. This, in turn, follows by applying recursively the following claim: if $\delta \phi =0 $ and 
\begin{equation}\label{eq:phi} 
\phi(\xi_0,..,\xi_{p-1}) = 0,
\end{equation} 
for all $(\xi_0,..,\xi_{p-1})\in B_p\V$ such that $\xi_i = 0_{g_i}, i=0,..,l\leq p-1$ then there exists a $\psi \in C^\infty_{1\mbox{-}\hom}(B_{p-1}\V)$ such that $\phi+\delta \psi$ satisfies \eqref{eq:phi} for all $(\xi_0,..,\xi_{p-1})\in B_p\V$ such that $\xi_i = 0_{g_i}, i=0,..,l-1$. Notice that for $l=p-1$, eq. \eqref{eq:phi} follows from $\phi$ being homogeneous of degree $1$. To prove this claim for $l<p-1$, one choses any $\psi \in C^\infty_{1\mbox{-}\hom}(B_{p-1}\V)$ such that
\[ \psi(\xi_1,..,\xi_{p-1}) = \phi(0_{\pi(\xi_{p-1})^{-1}\cdot \cdot \cdot \pi(\xi_{1})^{-1}}, \xi_1, ... \xi_{p-1}) \] for all $(\xi_1,..,\xi_{p-1})\in B_{p-1}\V$ such that $\tar_\V(\xi_1) = 0_{\tar(\pi(\xi_1))}$. This is always possible since the subset of such elements in $B_{p-1}\V$ is a smooth embedded submanifold since the target map is a submersion. What needs to be shown now is 
\[ (\phi + \delta \psi)(\xi_0,..,\xi_{p-1}) = 0, \ \forall (\xi_0,..,\xi_{p-1})\in B_p\V : \  \xi_i = 0_{g_i}, i=0,..,l-1.\]
Finally, this last identity follows by evaluating
\[ 0=(\delta \phi)(0_{\pi(\xi_{p-1})^{-1}\cdot \cdot \cdot \pi(\xi_{0})^{-1}}, \xi_0,..,\xi_{p-1}),\]
and using the recursion hypothesis.
\end{proof}

For a \vba $\v \to A$, the \emph{VB-algebroid cochain complex} is defined exactly as the complex of $1$-homogeneous cochains 
$$
CE^k_{VB}(\v):=CE^k_{1\mbox{-}\hom}(\v).
$$
The restriction of the Van Est map to $1$-homogeneous cochains as in Section \ref{sec:vanest} provides a chain map $\VE_{1\mbox{-}\hom}: C^\bullet_{1\mbox{-}\hom}(\V) \to CE^\bullet_{VB}(\v)$. Its restriction to the subcomplex $C^\bullet_{VB}(\V) \subset C^\bullet_{1\mbox{-}\hom}(\V)$ defines a chain map that we shall denote by $$\VE_{VB}: C^\bullet_{VB}(\V)\to CE^\bullet_{VB}(\v).$$

\begin{corollary}\label{cor:van_VB}
With the notations above, the Van Est map
$$
\VE_{VB} : H^\bullet(C_{VB}(\V)) \to H^\bullet(CE_{VB}(\v))
$$
is a right-module homomorphism over $\VE_\G: H^\bullet(\G) \to H^{\bullet}(\g)$. Moreover, 
if $\G$ is source $p_0$-connected, then $\VE_{VB}$ is an isomorphism in degree $p$, for all $p\leq p_0$ and it is injective for $p=p_0+1$.
\end{corollary}

\paragraph{Cohomological vanishing for proper groupoids}
The VB-groupoid cohomology can be shown to be trivial in several cases as shown by the following proposition.
\begin{proposition}\label{prop:cohomovan}
When $\G$ is a proper groupoid or, more generally, admits a Haar system $d\mu$ together with a cut off function $c\in \C(M)$ (see, e.g., \cite{AC3} and the proof below), then $$H^p(C^\bullet_{VB}(\V))=0, \ p\geq 2.$$
\end{proposition}
\begin{proof}
The idea is to define a map $C_{VB}^p(\V) \ni \phi \mapsto \kappa(\phi)\in C_{VB}^{p-1}(\V), p\geq 2$ by the formula 
$$ 
\kappa(\phi)(\xi_1,..,\xi_{p-1}) = \int_{\tar^{-1}(\sour(g_{p-1}))} \phi(\xi_1,..,\xi_{p-1},\sigma(h,\sour_\V(\xi_{p-1}))) \ c(s(h)) \ d\mu(h),
$$
where $g_i=\pi(\xi_i)\in\G,i=1,..,p-1$ as before and $\sigma: \tar^*E \to \V$ is any linear splitting of the epimorphism $\tar_\V:\V \to \tar^*E$. Notice that the r.h.s. in the formula above is independent of the choice of $\sigma$ since $\phi$ only depends on $(g_1,..,g_{p-1},h)$ and $\xi_1$. The key point is that, for $\delta \phi = 0, \phi \in C^p_{VB}(\V), p\geq 2,$ then $\delta \kappa(\phi) =(-1)^p \phi$,  hence leading to the above cohomological vanishing. This statement can be checked by direct computation: let us denote $\xi_{p+1}(h) = \sigma(h,\sour_\V(\xi_{p}))$ for $h\in \tar^{-1}(\sour(g_{p}))$ and $\eta_p(k)=\sigma(k,\sour_\V(\xi_{p-1}))$ for $k\in \tar^{-1}(\sour(g_{p-1}))$, then
\begin{align*}
 \delta \kappa(\phi) (\xi_1,..,\xi_p) & =  \int_{\tar^{-1}(\sour(g_{p}))}\left[ \phi(\xi_2,..,\xi_{p},\xi_{p+1}(h))+\sum_{i=1}^{p-1}(-1)^i \phi(\xi_1,.., \xi_i\xi_{i+1},..,\xi_p,\xi_{p+1}(h)) \right] \ c(s(h)) \ d\mu(h) \\
 & \hspace{-20pt} + (-1)^p \int_{\tar^{-1}(\sour(g_{p-1}))} \phi(\xi_1,..,\xi_{p-1},\eta_p(k)) \ c(s(k)) \ d\mu(k) \\
 & =  (-1)^p \int_{\tar^{-1}(\sour(g_{p}))} [- \phi(\xi_1,..,\xi_{p-1},\xi_p\xi_{p+1}(h)) + \phi(\xi_1,..,\xi_p) ]\ c(s(h)) \ d\mu(h)  \\
  & \hspace{-20pt} + (-1)^p \int_{\tar^{-1}(\sour(g_{p-1}))} \phi(\xi_1,..,\xi_{p-1},\eta_p(k)) \ c(s(k)) \ d\mu(k) \\
  &= (-1)^p \phi(\xi_1,..,\xi_p).
\end{align*}
Above, the first equality follows from the definitions of $\delta$ and $\kappa$, the second equality follows by applying $\delta \phi = 0$ inside the square brackets and, finally, the third equality follows by the normalization condition $\int_{\tar^{-1}(x)} c(s(h)) \ d\mu(h) = 1$ and by the left invariance of the measure $\int_{\tar^{-1}(s(g))} f(gh) \ d\mu(h) = \int_{\tar^{-1}(\tar(g))} f(k) \ d\mu(k)$ together with the independence of $\phi(\xi_1,..,\xi_p)$ on the $\xi_j$'s for $j>1$, as was mentioned before.
\end{proof}

Let us now mention an application of the above general vanishing result, following \cite{CrMo}.
Given a Lie algebroid $\g \to M$, there exists a complex $C^\bullet_{\mathrm{def}}(\g)$ controlling the deformations of $\g$ and which is related to VB-cohomology as follows. Consider the induced linear Poisson structure on $\g^*$, $\pi \in \Gamma(\wedge^2 T\g^*)$. The cotangent Lie algebroid $T^*\g\to \g^*$ has the property that its Chevalley-Eilenberg complex  $(\CE(T^*\g),  d)$ is isomorphic to the Poisson complex $(\mathfrak{X}(\g^*), [\pi,\cdot])$ (see \cite{Mac-Xu}). Under this isomorphism, the subcomplex $\CE^\bullet_{VB}(T^*\g)\subset \CE^\bullet(T^*\g)$ corresponds to the so called {\em linear Poisson complex} $\mathfrak{X}_{\ell in}(\g^*)$ of $\g^*$. On the other hand, Proposition 7 in \cite{CrMo} shows that $\mathfrak{X}^\bullet_{\ell in}(\g^*) \cong C^\bullet_{\mathrm{def}}(\g)$, so that
$$
\CE^\bullet_{VB}(T^*\g)\cong \mathfrak{X}^\bullet_{\ell in}(\g^*) \cong C^\bullet_{\mathrm{def}}(\g).
$$  
On the groupoid side, for $\G\rightrightarrows M$ a Lie groupoid,  the complex $C_{VB}(T^*\G)$  was shown in \cite{CrMeSt} to be isomorphic to the complex $C_{\mathrm{def}}(\G)$ controlling deformations of the Lie groupoid structure. 

In this context, Corollary \ref{cor:van_VB} recovers a result from \cite{CrMeSt}: the map
$$
\VE_{\mathrm{def}}: H^\bullet_{\mathrm{def}}(\G) \to H^\bullet_{\mathrm{def}}(\g)
$$
defines a (graded) module homomorphism covering $\VE_\G: H^\bullet(\G) \to H^{\bullet}(\g)$ which induces isomorphisms in degrees $p \leq p_0$ and monomorphism in degree $p = p_0+1$ when $\G$ is source $p_0$-connected. 

By combining this result with our general vanishing criteria (Proposition \ref{prop:cohomovan} above), we further obtain an independent proof of the (cohomological) rigidity conjecture of \cite{CrMo}: if $\G$ is proper and source 2-connected, then $H^2_{\mathrm{def}}(\g)=0$. Note that the map $\VE_{\mathrm{def}}$ is the ``\textit{lin}-version'' of the Van Est map which was assumed to exist in \cite{CrMo} as a step towards proving their conjecture. 

\begin{remark}
The conjecture was originally proved in \cite{AS} using a Van Est result for representations up to homotopy. In particular, they use a vanishing result for cohomologies with coefficients in representations up to homotopy established in \cite{AC3}. Our vanishing result should be considered as a geometric counterpart to theirs in the 2-term case (see below).  
\end{remark}


\subsection{Splittings and representations up to homotopy}
VB-groupoids and VB-algebroids can be (non-canonically) \emph{split} into the base Lie groupoid/algebroid data and representation-like information on the fibers (recall examples \ref{ex:repG}, \ref{example:rep}). It turns out that the correct notion encoding this split data is that of a ($2$-term) \emph{representations up to homotopy} (\cite{AC2, AC3,GM08,GM10}), which we now recall.

Let $\G \rightrightarrows M$ be a Lie groupoid with Lie algebroid $\g \to M$ and $\E = C[1] \oplus E$ a graded vector bundle over $M$ with $C$ in degree $-1$ and $E$ in degree $0$.
The associated space of $\E$-valued (normalized) $p$-cochains is defined as
\comment{Change $\mu_{-1}, \mu_0$ to $\mu_C, \mu_E$ and $\E=\E_{-1}\oplus E_0$ to $C[1]\oplus E$.}
$$
C(\G,\E)^p := \{\mu:=(\muE, \muC) \in \Gamma(B_p\G; \tar_p^*E) \oplus \Gamma(B_{p+1}\G; \tar_{p+1}^*C) \,\, | \,\, s_i^* \muE = 0, s_i^*\muC=0\},
$$
where $s_i: B_{\bullet} \G \to B_{\bullet+1} \G$ is the $i$-th degeneracy map. There is a (right) $C^{\bullet}(\G)$-module structure on $C(\G,\E)^{\bullet}$ defined by 
$
\mu \star f = (\muE \star f, \muC \star f),
$
where each component is given by formula \eqref{star_prod}.
A representation up to homotopy of $\G$ on $\E$ is a $\R$-linear map $\D_\G: C(\G,\E)^\bullet \to  C(\G,\E)^{\bullet+1}$ satisfying
$\D_\G^2=0$ and
$$
\D_\G(\mu \star f) = \D_\G(\mu) \star f + (-1)^{p} \mu \star (\delta f), \,\, \mu \in C(\G, \E)^p, f \in C^{p'}(\G).
$$
The resulting cohomology is denoted by $H(\G,\E)$. Note that $\star$ defines a right $H(\G)$-module structure on $H(\G, \E)$.

A representation up to homotopy on $\E$ can be alternatively given by quasi-actions $\Delta^{E}, \Delta^{C}$ of $\G$ on $E$ and $C$, respectively, a bundle map $\partial: C \to E$ and a smooth correspondence which, for each $(g_1, g_2) \in B_2 \G$, gives a linear map $\Omega_{(g_1,g_2)}: E|_{\sour(g_2)} \to C|_{\tar(g_1)}$, satisfying certain structural equations (see \cite{AC3, GM10}).
Moreover, in analogy with the case of an ordinary representation (c.f. Example \ref{ex:repG}), a representation up to homotopy of $\G$ on $\E$ endows $\V = \sour^*E^* \oplus_{\G} \tar^*C^* \toto C^*$ with a VB-groupoid structure \cite{GM10}. The structure maps are given by
\begin{equation}\label{vb:struct}
\begin{array}{rl}
\sour_\V(\xi, g, \eta)& \!\! = (\Delta^{C}_g)^*\xi - \partial^*\eta, \,\,\tar_\V(\xi, g, \eta) = \xi, \,\,\,\,\, \xi \in C^*|_{\tar(g)},\, \eta \in E^*|_{\sour(g)}\\
(\xi_1, g_1, \eta_1)\cdot(\xi_2, g_2, \eta_2)  & \!\! =(\xi_1, \,g_1g_2,\, \Omega_{(g_1,g_2)}^*\xi_1 + (\Delta_{g_2}^{E})^*\eta_1 + \eta_2),\\
\end{array}
\end{equation}
for compatible arrows and $\1_\V(\xi) = (\xi, 1_m, 0)$, for $\xi \in C^*|_m$. Finally, in \cite{GM10} the authors show that every VB-groupoid can be presented (non-canonically) in this form, thus establishing a correspondence between VB-groupoids and $2$-term representations up to homotopy of $G$.

The above correspondence between VB-groupoid structures and representations up to homotopy can be understood from the following relation between the cochain complex associated to $\E$ and that of $1$-homogeneous cochains on $\V$. Consider the map
$
\Psi: C(\G, \E)^p  \to C^\infty_{1\mbox{-}\hom}(B_{p+1}\V)
$
defined by 
\begin{equation}\label{def:phi}
\Psi(\mu)((\xi_1,g_1, \eta_1),..,(\xi_{p+1}, g_{p+1}, \eta_{p+1})) = \langle \eta_1, \muE(g_2,..,g_{p+1}) \rangle + \langle \xi_1,\muC(g_1,..,g_{p+1}) \rangle.
\end{equation}
In \cite{GM10} (see  Theorem $5.6$), it is proven that $\Psi: C(\G, \E)^{\bullet} \to C_{1\mbox{-}\hom}^{\bullet+1}(\V)$ is a monomorphism of graded $C(\G)$-modules satisfying
\[ \Psi \circ (-\D_\G) = \delta \circ \Psi\]
whose image coincides with the VB-groupoid cochain complex $C_{VB}^\bullet(\V) \subset C_{1\mbox{-}\hom}^{\bullet}(\V)$ (shifted by one, hence the minus sign in the eq. above).
We then obtain, as a direct consequence of Lemma \ref{lma:vbgc0},
\begin{lemma}\label{lma:VBG_cohom}
The induced map in cohomology $\Psi: H^p(\G,\E) \to H^{p+1}_{1\mbox{-}\hom}(\V)$ is an isomorphism of right $H^\bullet(\G)$-modules for all $p$.
\end{lemma}

\paragraph{The infinitesimal counterpart.}
Let $\g$ be a Lie algebroid, $\E$ as before, and consider
$$
\Omega(\g, \E)^p = \Gamma(\Lambda^p \g^*\otimes E) \oplus \Gamma(\Lambda^{p+1} \g^* \otimes {C}).
$$
The space $\Omega(\g, \E)$ is a right $\Gamma(\Lambda^\bullet \g^*)$-module with multiplication defined by wedge product on the right on the $\Lambda^\bullet \g^*$ factor. A representation up to homotopy of $\g$ on $\E$ is a $\R$-linear map $\D_\g: \Omega(\g, \E)^{\bullet} \to \Omega(\g, \E)^{\bullet+1}$ satisfying $\D_\g^2=0$ and 
$
\D_\g(\omega \wedge \beta) = \D_\g(\omega) \wedge \beta + (-1)^p \omega \wedge d \beta, \, \,  \omega \in \Omega(\g, \E)^p, \, \beta \in \Gamma(\Lambda \g^*).
$
We denote the cohomology of $(\Omega(\g, \E), \D_\g)$ by $H(\g, \E)$.

As in the VB-groupoid case, VB-algebroid structures on  $\v= C^* \times_M \g \times_M E^* \to C^*$ are in $1:1$ correspondence with representations up to homotopy of $\g$ on $\E = C[1]\oplus E$ (see \cite{GM08}). We recall here how this correspondence can be seen from the cohomological perspective. The space of sections $\Gamma(C^*, \v)$ is generated, as a $C^\infty(C^*)$-module, by sections:
\begin{align*}
\chi_u(\xi) = (\xi, u(m),0),  \,\,\Upsilon_\eta(\xi) = (\xi, 0, \eta(m)),
\end{align*}
for $\xi \in C^*|_m, u \in \Gamma(\g), \eta \in \Gamma(E^*)$.  Define a map 
\begin{equation}\label{evmap}
ev: \CE^{p+1}_{1\mbox{-}\hom}(\v) \to \Omega(\g, \E)^{p}, \ ev(\alpha) = (\hat{\alpha}_E, \hat{\alpha}_{C}),
\end{equation}
where $\hat{\alpha}_E \in \Gamma(\Lambda^p \g^*\otimes E)$ and $\hat{\alpha}_{C} \in \Gamma(\Lambda^{p+1} \g^*\otimes {C})$, by
\begin{align*}
\<\hat{\alpha}_E(u_1, \dots, u_p), \eta\> & = \alpha(\Upsilon_\eta, \chi_{u_1}, \dots, \chi_{u_p}) \in C^{\infty}_{0\mbox{-}\hom}(C^*) \cong C^\infty(M) \\
\hat{\alpha}_{C}(u_1, \dots, u_{p+1}) & = \alpha(\chi_{u_1}, \dots, \chi_{u_{p+1}}) \in C^{\infty}_{1\mbox{-}\hom}(C^*) \cong \Gamma(C),
\end{align*}
for $u_1, \dots, u_{p+1} \in \Gamma(\g), \eta \in \Gamma(E^*)$.

\begin{lemma}\label{lemma:ev}
Under the identification $\Gamma(\Lambda^\bullet \g^*) \cong CE_{hom,0}(\v)$, the map $ev$ is a (right) $\Gamma(\Lambda^\bullet \g^*)$-module isomorphism. 
\end{lemma}

\begin{proof}
Let $\{\xi^k\}_{k=1}^{\mathrm{rank}(C^*)}$, $\{\gamma^j\}_{j=1}^{\mathrm{rank}(\g^*)}$ and $\{e_i\}_{i=1}^{\mathrm{rank}(E)}$ be local frames for $C^*$, $\g^*$ and $E$ respectively. We shall identify $e_i$ (resp. $\gamma^j$) with the corresponding section of $\v^*$: $C^*|_m \ni \xi \mapsto (\xi, 0, e_i(m))$ (resp. $\xi \mapsto (\xi, \gamma^j(m), 0)$). Locally, any element of $\alpha \in CE_{1\mbox{-}\hom}^{p+1}(\v)$ is written
$$
\alpha(m, \xi) = a_k A_{j_1 \dots j_{p+1}}^k(m) \gamma^{j_1} \wedge \dots \gamma^{j_{p+1}} + B_{j_1 \dots j_{p}}^i(m) \, e_i \wedge \gamma^{j_1} \wedge \dots \gamma^{j_{p}},
$$
where $\xi = a_k \,\xi^k(m)$. From the definition, one sees that
\begin{align*}
A_{j_1 \dots j_{p+1}}^k(m) & = \<\hat{\alpha}_{C}(u_{j_1}, \dots, u_{j_{p+1}}), \xi^k(m)\>\\
B_{j_1 \dots j_{p}}^i(m) & = \<\hat{\alpha}_{E}(u_{j_1, \dots, u_{j_p}}), \eta^i(m)\>,
\end{align*}
where $\{u_j\}$, $\{\eta^i\}$ are local frames for $\g$ and $E^*$ dual to $\{\gamma^j\}$, $\{e_i\}$, respectively. It is now straightforward to prove the statement.
\end{proof}

Hence, the operator $\D_\g$ defined by $\D_\g \circ ev= ev \circ (-d)$, where $d$ is the Chevalley-Eilenberg differential of $\v$, defines a representation up to homotopy of $\g$ on $\E$. (Note that $ev$ shifts degree by minus one, hence the sign in the definition of $\D_\g$.) It is shown in \cite{GM08} that, moreover, every VB-algebroid can be split as $\v\simeq \C^* \times_M \g \times_M E^* \to C^*$, thus establishing a correspondence between VB-algebroids and $2$-term representations up to homotopy of $\g$.

Given a representation up to homotopy $\D_\G: C(\G, \E) \to C(\G, \E)$ of $\G$ on $\E$, the VB-groupoid $\V \toto C^*$ defined by \eqref{vb:struct}, seen as a Lie groupoid over $C^*$, has a Lie algebroid whose underlying bundle is precisely $\v= C^* \times_M \g \times_M E^* \to C^*$. In this case, the above construction of $\D_\g$ can understood as the differentiation of the representation $\D_\G$, namely, $\D_\g= Lie(\D_\G)$. (See also \cite{AS}.)

\begin{remark}
A representation up to homotopy of $\g$ on $\E$ can be alternatively described by a map $\partial: C \to E$, $\g$-connections $\nabla^{E}, \nabla^{C}$ on ${E}$ and ${C}$, respectively and a curvature term $R \in \Gamma(\Lambda^2 \g^*\otimes Hom({E},{C}))$ satisfying some compatibility equations (see \cite{AC2, GM08}). We refer to \cite{BCO} for the formulas of the operators $(\partial, \nabla^{E}, \nabla^{C}, R)$ corresponding to $Lie(\D_\G)$ in terms of the data defining $D_\G$.
\end{remark}

\subsection{The Van Est theorem for representations up to homotopy}
Let us define $\VE_{rep}: C(\G, \E)^{p} \to \Omega(\g, \E)^p$ by $\VE_{rep}:= ev \circ VE_{1\mbox{-}\hom} \circ \Psi$. Diagramatically,
\begin{equation}
\begin{xy}
(0,35)*+{C(G, \E)^k}="A"; (40,35)*+{C_{1\mbox{-}\hom}^{\infty}(\V^{(k+1)})}="B"; 
(0,20)*+{\Omega(\g, \E)^k}="C"; (40,20)*++{\CE_{1\mbox{-}\hom}^{k+1}(\v)}="D";
{\ar@{->}^{\Psi} "A"; "B"}
{\ar@{->}^{\VE_{1\mbox{-}\hom}}"B";"D"}
{\ar@{<-}^{ev}"C";"D"}
{\ar@{->}_{\VE_{rep}}"A";"C"}
\end{xy}
\end{equation}
It is clear from the previous discussion that $\VE_{rep}$ is a chain map.

\begin{theorem}\label{thm:rep}
The Van Est map $\VE_{rep}: H^\bullet(\G, \E) \to H^\bullet(\g, \E)$ is a right module homomorphism over $\VE_\G: H^\bullet(\G) \to H^\bullet(\g)$. Moreover, if $\G$ is source $p_0$-connected, then the induced map in cohomology $\VE_{rep}: H^p(\G, \E) \to H^p(\g, \E)$ is an isomorphism for $-1 \leq p \leq p_0-1$ and it is injective for $p=p_0$.
\end{theorem}

\begin{proof}
This is a straightforward consequence of Theorem \ref{main}, Lemmas \ref{lma:VBG_cohom} and \ref{lemma:ev}. Notice the shift in grading for which one has isomorphisms. This arises because one has to apply Thm. \ref{main} to $C_{1\mbox{-}\hom}^{\infty}(B_{k+1}\V) \to CE_{1\mbox{-}\hom}^{k+1}(\v)$ in order to analyse $C(\G,\E)^k \to \Omega^k(\g, \E)$.
\end{proof}

The fact that the above cohomology groups are isomorphic was also proven in \cite{AS} using different techniques (in the more general setting of representations on arbitrarily graded vector bundles). Notice that, from our perspective, it just arises as a refinement of the usual Van Est map for $\V$ for $1$-homogeneous cochains.
%
%
 \br \emph{(Formulas for $\VE_{rep}$)}
For $u \in \Gamma(\g)$, define $R_u: C^p(\G, \E) \to C^{p-1}(\G, \E)$ by:
$$
(R_u \muC)(g_1, \dots, g_{p}) = \left.\frac{d}{d\epsilon}\right|_{\epsilon=0} \Delta^{C}_{\phi_\epsilon^u(\tar(g_1))^{-1}} \muC(\phi_\epsilon^u(\tar(g_1)), g_1, \dots, g_p),
$$
where $\phi^u_\epsilon: M \to \G$ is the flow of the right-invariant vector field $\overrightarrow{u}$ and the definition $R_u\muE$ is analogous. Note that our conventions are different from \cite{AS}. One can now check the identities:  
\begin{align*}
R_{\chi_u} \Psi(\mu) & = \Psi(R_u \muC, 0), \,\, R_{\Upsilon_\eta}\Psi(\mu) = q^*\<\mu_E, \eta\>\\
R_{\chi_v}R_{\Upsilon_\eta}\Psi(\mu)  & = q^*\<R_v\mu_E, \eta\>, \, R_{\Upsilon_\eta} R_{\chi_v} \Psi(\mu)= 0,
\end{align*}
where $q: B_{\bullet}\V \to B_{\bullet} \G$ is the projection map.  
Using these identities, it is now straightforward to check that $VE_{rep}(\mu) = (\hat{\mu}_E, \hat{\mu}_{C}) \in \Gamma(\Lambda^p \g^*\otimes {E}) \oplus \Gamma(\Lambda^{p+1} \g^*\otimes {C})$ is given by
\begin{align*}
\hat{\mu}_E(u_1, \dots, u_p) & = (-1)^p\sum_{\sigma \in S_p} sgn(\sigma) R_{u_{\sigma(1)}} \dots R_{u_{\sigma(p)}} \mu_0,\\
\hat{\mu}_{C}(u_1, \dots, u_{p+1}) & =  \sum_{\sigma \in S_{p+1}} sgn(\sigma) R_{u_{\sigma(1)}} \dots R_{u_{\sigma(p+1)}} \muC.
\end{align*}
\er


\section{Differential forms with values in a representation}\label{sec:forms}
In this Section, we study differential forms on a Lie groupoid $\G$ with values in a representation $C \to M$. These objects were introduced in \cite{CSS} together with their infinitesimal counterparts, the Spencer operators. We here  provide a Van Est theorem for them as an application of our main result. The key point is the idea of seeing forms as homogeneous functions.

\subsection{The Van Est theorem for differential forms with coefficients}
We start this section by formally defining the ingredients entering the Van Est theorem for forms with coefficients (Theorem \ref{diff_vanest} below) without any reference to the VB-groupoids and algebroids. Later, we shall show how VB-groupoids and algebroids provide a useful framework to interpret many of the definitions and to give a proof of Theorem \ref{diff_vanest}.

Let $\G \toto M$ be a Lie groupoid, $C \to M$ be a representation of $\G$ and consider the map $\tar_p: B_p\G \to M$, $\tar_p(g_1, \dots, g_p) = \tar(g_1)$. When no confusion arises, we shall omit the reference to $p$ and simply denote $\tar_p$ by $\tar$. The space of $q$-differential forms on $\G$ with coefficients in $C$ is the complex $\Omega^q(B_\bullet\G, \tar^*C)$. It carries a differential $\delta: \Omega^q(B_{p-1}\G, \tar^*C) \to \Omega^{q}(B_{p}\G, \tar^*C)$ defined by
\begin{align*}
\delta \omega|_{(g_1,\dots, g_p)} & = \Delta_{g_1} \circ \partial_0^*\omega + \sum_{i=1}^{p} (-1)^{i}\partial_i^*\omega, \,\,\, \text{ for } p \geq 2 \\
\delta\omega|_g  & = \tar^*\omega -  \Delta_g \circ \sour^* \omega, \,\, \text{ for } p=1.
\end{align*}
It is straightforward to check that $\delta^2=0$. 

Note that, for $\omega \in \Omega^q(\G, \tar^*C)$,
\begin{align*}
 \,\,\delta\omega|_{(g_1,g_2)}  = \Delta_{g_1} \circ \pr_1^*\omega - m^*\omega + \pr_2^*\omega
\end{align*}
where $\pr_i(g_1,g_2) = g_i$, for $i=1,2$. In this case, a form $\omega \in \Omega^q(\G, \tar^*C)$ which satisfies $\delta \omega =0$ is called \textit{multiplicative} (see \cite{CSS}). Note that $\Omega^{q}(B_\bullet\G, \tar^*C)$ is a right dg-module for $C^\bullet(\G)$ with the module structure defined as usual by:
$$
(\omega \star f)|_{(g_1,\dots, g_{p+p'})} = \omega|_{(g_1,\dots, g_p)} f(g_{p+1},\dots, g_{p+p'}), \,\,\, \omega \in \Omega^q(B_p\G, \tar^*C),\,\, f \in C^{p'}(\G). 
$$

\begin{remark}
In the case of trivial coefficients (i.e. when $C$ is the trivial line bundle),  the de Rham differential turns $\Omega^q(B_p\G, t^*C) = \Omega^q(B_p\G)$ into a double complex known as the Bott-Shulman double complex associated to $\G$ (see \cite{AC} and references therein). In the remaining of this paper, we shall focus on the cohomology of $\delta$ alone and leave the investigation of compatible double complex structures (corresponding to 'multiplicative linear flat connections')  for future work.
\end{remark}

Let $\g \to M$ be the Lie algebroid of $\G$. Similarly to \cite{AC}, we define the Weil complex $W^{p,q}(\g, C)$ to be the space of sequences $c=(c_0,c_1, \dots)$, where each
$$
c_k: \underbrace{\Gamma(\g) \times \dots \times \Gamma(\g)}_{(p-k)-times} \rightarrow \Omega^{q-k}(M, S^k\g^* \otimes C)
$$
is a $\R$-linear skew-symmetric map whose failure at being $C^{\infty}(M)$-linear is controlled by
\begin{equation}\label{leibniz}
c_k(f u_1, \dots, u_{p-k}|\cdot) = f c_k(u_1, \dots, u_{p-k}| \cdot) + df \wedge c_{k+1}(u_2,\dots, u_{p-k}|u_{1},\cdot), \,\, \forall \, f \in C^{\infty}(M).
\end{equation}
For each $q$, the complex $W^{\bullet, q}(\g, C)$ carries a differential $\mathrm{d}_W: W^{p,q}(\g, C) \to W^{p+1,q}(\g, C)$ which we now define. First, note that $\Omega^i(M, S^j\g^*\otimes C)$ is a module for the Lie algebra $\Gamma(\g)$. Indeed, for $\alpha \in \Omega^i(M), \, P \in \Gamma(S^j\g^*\otimes C)$,
$$
u \cdot (\alpha \otimes P)  = (\Lie_{\rho(u)}\alpha) \otimes P + \alpha \otimes (u \cdot P), \,\, u \in \Gamma(\g)
$$
defines an action of $\Gamma(\g)$ on $\Omega^i(M, S^j\g^*\otimes C)$, where
$$
(u \cdot P)(v_1, \dots, v_k) = \nabla_{u} P(v_1, \dots, v_k) - \sum_{i=1}^{k} P(v_1, \dots, [u, v_i], \dots, v_k),
$$
and $\nabla: \Gamma(\g)\times \Gamma(C) \to \Gamma(C)$ is the $\g$-connection giving the representation $C$. Now, $\mathrm{d}_W$ is defined by
\begin{align}
\label{def:d_W} \hspace{-30pt}\mathrm{d}_W(c)_k (u_1, \dots, u_{p-k+1}|v_1, \dots, v_k) & = (-1)^k\left(d_{CE}(c_k)(u_1, \dots, u_{p-k+1}| v_1, \dots, v_k)\vphantom{\sum_{j=1}^k}\right.\\
\nonumber & \hspace{-30pt} \left.- \sum_{j=1}^k\,i_{\rho(v_j)} c_{k-1} (u_1, \dots, u_{p-k+1}| v_1, \dots, \hat{v_j}, \dots, v_k)\right)
\end{align}
where $d_{CE}$ is Chevalley-Eilenberg differential on the complex $C^{\bullet}(\Gamma(\g), \Omega^{q-k}(M, S^k\g^*\otimes C))$.
\comment{The signs in the differential here changed because of the order convention.}
There is a right $\Gamma(\Lambda^\bullet \g^*)$-module structure on $W^{\bullet,q}(\g, C)$. It is defined, for $\beta \in \Gamma(\wedge^{p'} \g^*)$ and $c \in W^{p,q}(\g, C)$ by
$$
(c \wedge \beta)_k (u_1,\dots, u_{p+p'-k}|\cdot) = \!\!\!\!\sum_{\sigma\in S(p-k,p')} \!\!\!\!\!sgn(\sigma) c_k(u_{\sigma(1)}, \dots, u_{\sigma(p-k)})\,\beta(u_{\sigma(p-k+1)}, \dots, u_{\sigma(p+p'-k)}),
$$
where $S(p-k,p')$ is the space of $(p-k,p')$-unshuffles. 

\begin{proposition}\label{Weil:dg}
$W^{\bullet,q}(\g,C)$ is a right dg-module for $\Gamma(\Lambda^\bullet \g^*)$.
\end{proposition}

This result will follow from a evaluation isomorphism similar to \eqref{evmap} (see Proposition \ref{prop:evaluation} below) between $W^{\bullet, q}(\g,C)$ and another right dg-module for $\Gamma(\Lambda^\bullet \g^*)$. It is important to remark that all the signs appearing in the above formula for $\mathrm{d}_W$, as well as in formula \eqref{c_omega} below, are natural consequences of a simple ordering convention in the definition of this evaluation isomorphism.  
\comment{Remark about our sign conventions.}

\begin{remark}
\noindent 
\begin{itemize}
\item $W^{0,q}(\g, C)=\Omega^q(M,C)$. In this case, for $c \in W^{0,q}(\g,C)$, $\mathrm{d}_W(c)=(\mathrm{d}_W(c)_0,\mathrm{d}_W(c)_1)$ where $\mathrm{d}_W(c)_0: \Gamma(\g) \rightarrow \Omega^q(M,C)$, $\mathrm{d}_W(c)_1 \in \Omega^{q-1}(M, \g^*\otimes C)$ are given by
$$
\mathrm{d}_W(c)_0(u) = u \cdot c, \, \mathrm{d}_W(c)_1(v) = i_{\rho(v)} c.
$$
\item For $W^{1,q}(\g,C)$, its elements are $c=(c_0,c_1)$, where $c_0: \Gamma(g) \rightarrow \Omega^q(M, C)$, $c_1 \in \Omega^{q-1}(M, \g^* \otimes C) \cong Hom(\g, \Lambda^{q-1}T^*M\otimes C)$. In this case, 
\begin{align*}
\mathrm{d}_W(c)_0(u_1,u_2) & = u_1 \cdot c_0(u_2) - u_2 \cdot c_0(u_1) - c_0([u_1,u_2])\\
\mathrm{d}_W(c)_1(u|v) & = i_{\rho(v)}c_0(u) - u \cdot c_1(v) + c_1([u,v])\\
\mathrm{d}_W(c)_2(v_1, v_2) & = - i_{\rho(v_1)} c_1(v_2)  - i_{\rho(v_2)} c_1(v_2).
\end{align*}
\end{itemize}
Note that, in the case $p=1$, the equation $\mathrm{d}(c) = 0$ is equivalent to $(c_0,c_1)$ being a $C$-valued Spencer operator on $\g$ \cite{CSS} and, thus, in particular, to $(c_0,c_1)$ being an infinitesimally multiplicative form \cite{AC} when $C=\R$, with the trivial representation.
\end{remark}

\paragraph{The Van Est map.}
Given $u \in \Gamma(\g)$, let $\phi^u_\epsilon: \G \to \G$ be the flow of the right-invariant vector field $\overrightarrow{u}$. The flow of the corresponding vector field $B_pu \in \mathfrak{X}(B_p\G)$ is given by
$$
\psi^u_{\epsilon}(g_1, \dots, g_p) = (\phi^u_\epsilon(g_1), g_2, \dots, g_p).
$$
Define operators $R_u: \Omega^q(B_p\G, \tar^*C) \rightarrow \Omega^q(B_{p-1}\G, \tar^*C)$, $J_u: \Omega^q(B_p\G, \tar^*C) \rightarrow \Omega^{q-1}(B_{p-1}\G, \tar^*C)$ by
\begin{align}
R_u\omega|_{(g_1, \dots, g_{p-1})} & = s_0^*\left(\left.\frac{d}{d\epsilon}\right|_{\epsilon=0} \Delta_{\phi^u_{\epsilon}(\tar(g_1)))^{-1}} \circ \psi^{u \, *}_{\epsilon} \omega \right)\nonumber \\
J_u \omega   & = s_0^*\,i_{B_pu} \omega.\label{eq:RJ}
\end{align}

The Van Est map $VE_\Omega: \Omega^q(B_p\G, \tar^*C) \to W^{p,q}(\g, C)$,
$
VE_\Omega(\omega) = (c_0(\omega), c_1(\omega), \,\dots),
$
has each $c_k(\omega)$ given by
\begin{equation}\label{c_omega}
c_k(\omega)(u_1, \dots, u_{p-k}|v_1, \dots, v_k) = (-1)^{\frac{k(k-1)}{2}}\sum_{\sigma\in S(p)} sgn(\sigma) (-1)^{\epsilon(\sigma,k)} D_{\sigma(1)} \dots D_{\sigma(p)}\omega,
\end{equation}
where 
\begin{equation}\label{def:D}
D_{j} =
\begin{cases}
J_{v_{j}}, & \text{if }\,\, j \in \{1, \dots, k\},\\
R_{u_{j-k}}, & \text{if } \,\,j \in \{k+1, \dots, p\}
\end{cases}
\end{equation}
and 
\begin{equation*}
\epsilon(\sigma,k) = \# \{(i,j) \in \{1, \dots, k\}\times \{1, \dots, k\} \,\, |  \,\, i < j \text{ and } \sigma^{-1}(i) > \sigma^{-1}(j)\}.
\end{equation*}

\begin{theorem}\label{diff_vanest}
$VE_\Omega$ is a chain map and its induced map on cohomology $VE_\Omega: H^{\bullet}(\Omega^{q}(B_\bullet\G, \tar^*C)) \rightarrow H^\bullet(W^{\bullet,q}(\g,C))$ is a right module homomorphism over $VE_\G: H^\bullet(\G) \to H^\bullet(\g)$. Moreover, if $\G$ is source $p_0$-connected, then
$$
VE_\Omega: H^{p}(\Omega^{q}(B_\bullet\G, \tar^*C)) \rightarrow H^p(W^{\bullet,q}(\g,C))
$$
is an isomorphism for $p \leq p_0$ and it is injective for $p=p_0+1$, for each fixed $q$.
\end{theorem}


In the remainder of the paper, we will prove Theorem \ref{diff_vanest} by showing how it can be framed as a Van Est result for a class of VB-groupoids. Notice that the above theorem recovers Theorem 5.1 of reference \cite{AC} (up to some sign conventions) when $C=M\times \mathbb{R}$ with the trivial representation. It is interesting to note that, even in this particular case, our proof is independent of the one given in \cite{AC}.

\subsection{Forms as functions}\label{ex:tensor}
The key idea in the proof of \ref{diff_vanest} is that differential forms can be seen as homogeneous functions on a appropriate space. In this subsection, we shall ellaborate on this classical viewpoint.

Let $V_1, \dots, V_{q+1}$ be vector bundles over $B$ and consider the fiber-product $\prod_{j=1}^{q+1} V_j = V_1\times_B \dots \times_B V_{q+1}$ with the natural vector bundle structure over $B$ (the Whitney sum $V_1 \oplus \dots \oplus V_{q+1} \to B$). 

\paragraph{Simple functions.}
For $i=1, \dots q+1$, let $0_i: \prod_{j \neq i} V_j \to \prod_{j} V_j$ be the inclusion which puts a zero in the $i$-th coordinate. A function $f \in \C(\prod_{j} V_j)$ is said to be \textit{simple} if 
$$
0_i^*f=0, \,\forall \, i=1,\dots, q+1.
$$

For a subset $I \subset \{1, \dots, q+1\}$, denote by $|I|$ its cardinality and by $0_I: \prod_{j \notin I} V_j \to \prod_{j} V_j$ the inclusion which puts a zero in the entries indicated by the elements of $I$. Define $P_{(l)}: \C(\prod_j V_j) \to \C(\prod_jV_j)$, $l=-1,0,1, \dots, q$, by
\begin{equation}\label{proj:simple}
\begin{array}{rl}
P_{(-1)}(f) &  \hspace{-5pt}=  f\\
P_{(l)}(f)  &  \hspace{-5pt}=  \displaystyle P_{(l-1)}(f) - \sum_{|I|=q+1-l} 0_I^*P_{(l-1)}(f),\,\, \text{for} \,\,l=0, \dots, q.
\end{array}
\end{equation}
Each $P_{(l)}$, $l=0,\dots, q$, is a projection onto the space of functions of $\prod_j V_j$ which vanishes whenever $(q+1-l)$ entries are zero. In particular, $P_{\mathrm{spl}}: = P_{(q)}$ is a  projection onto the space of simple functions.

\paragraph{Multilinearity and skew-symmetry.}
The map
$$
\begin{array}{ccl}
\Gamma(B, V_1^*\otimes \dots \otimes V_{q+1}^*) & \longrightarrow & C^{\infty}(\prod_{j=1}^{q+1} V_j)\\
\mu_1 \otimes \dots \otimes \mu_{q+1} & \longmapsto &  (\ell_{\mu_1}\circ \pr_1) \cdots(\ell_{\mu_{q+1}} \circ \pr_{q+1}), \\ 
\end{array}
$$
is a monomorphism of $C^{\infty}(B)$-modules, where $\pr_i: \prod_{j=1}^{q+1} V_j \mapsto V_i$ is the projection onto the $i$-th summand. It follows from Taylor's Theorem that its image is the space of simple $q+1$-homogeneous functions.

We shall be mainly interested in the case $V_1=\dots=V_q=V$ and $V_{q+1}= W^*$ and we shall denote the $q$-fold fiber product $V\times_B \dots \times_B V$ by $\times_B^q V$. A function $f \in \C(\times_B^q V  \times_B W^*)$ is said to be \textit{skew-symmetric} if
$$
f(v_{\sigma(1)}, \dots, v_{\sigma(q)}, \xi) = sgn(\sigma) f(v_1, \dots, v_q, \xi), \forall \, v_i \in  V, \, \xi \in W^*, \, \sigma \in S_q.
$$

The map $P_{\sk}: \C(\times_B^q V \times_B W^*) \to \C(\times_B^q V \times_B W^*)$ defined by
\begin{equation}\label{proj:skew}
P_\sk(f) = \frac{1}{q!} \sum_{\sigma \in S_q} sgn(\sigma) f \circ \sigma,
\end{equation}
is a projection onto the space of skew-symmetric functions, where $S_q$ is the symmetric group and $\sigma: \times_B^q V \times_B W^* \to \times_B^q V \times_B W^*$ is the permutation of the first $q$ entries belonging to $V$ according to $\sigma$. Let us define

\begin{equation}\label{diff_iso}
\begin{array}{rcl}
\mathfrak{F}:\hspace{20pt}\Gamma(B, \Lambda^q V^* \otimes W) & \longrightarrow & C^\infty(\times_B^q V\times W^*) \vspace{5pt}\\
 \omega = (\mu_1 \wedge \dots \wedge \mu_q) \otimes \xi & \longmapsto & q! \,P_{\sk}\left(\,(\ell_{\mu_1}\circ \pr_1) \cdots (\ell_{\mu_q}\circ \pr_q) \, (\ell_{\xi}\circ \pr_{q+1})\right)\\
  & & 
\end{array}
\end{equation}
It is straightforward to check that $\f$ is a monomorphism of $C^\infty(B)$-modules whose image is the space of simple, skew-symmetric $q+1$-homogeneous functions. We shall denote the image of $\f$ by $C^\infty_\ext(\times_B^q V\times W^*)$. The projections $P_{\sk}, P_\spl$ and $P_{q+1\mbox{-}\hom}$ commute with each other, and so
\begin{equation}\label{proj:ext}
P_{\ext}:= P_\sk \circ P_\spl \circ P_{q+1\mbox{-}\hom} : \C(\times_B^q V \times_B W^*) \to C^\infty_{\ext}(\times_B^q V \times_B W^*) 
\end{equation}
is a projection onto $C^{\infty}_{\ext}(\times_B^q V \times_B W^*)$.

\begin{example}
For $V= B \times \R^n$, let $\{\theta_1, \dots, \theta_m\}$ be a local frame for $W$ and $\{e^1, \dots, e^n\}$ be a global frame for $V^*$. A point $p \in \times_B^q V\times_B W^*$, $q \leq n$,  has coordinates 
$$
p=(x,\underline{y_1}, \dots, \underline{y_q}, \xi_1, \dots, \xi_m), \,\, x \in B, \, \underline{y_j} = (y_{1,j},\dots, y_{n,j}) \in \R^n,  \xi_l \in \R.
$$
For a function $f \in C^\infty(\times_B^q V \times_B W^*)$, $P_{\ext}f = \frac{1}{q!}\f(\omega_f)$,  where $\omega_f \in \Gamma(B,\wedge^q V^* \otimes W)$ is given by
$$
\omega_f(p) =  \sum_{1 \leq k_1 < \dots < k_q \leq n} \sum_{i=1}^m \sum_{\sigma \in S_q} sgn(\sigma) \frac{\partial^{q+1}f}{\partial y_{k_{\sigma(1)}, 1}\dots {\partial y_{k_{\sigma(q)}, q}} \,\partial \xi_i}(x,0) \, e^{k_1} \wedge \dots \wedge e^{k_q} \otimes \theta_i(x)
$$
\end{example}

\subsection{The VB-groupoid behind the curtains.}
We here define the VB-groupoid whose differentiable cochain complex contain the complex of differential forms with coefficients. Later on, we show how the Weil complex is embedded in the Chevalley-Eilenberg complex of its Lie algebroid.  

\paragraph{Differential forms with coefficients.}
Let $T\G \toto TM$ be the tangent groupoid, obtained by taking the derivative of all the structure maps defining $\G$.  Let us introduce the \vbg $\bG_q \toto \bM_q$ defined by
\begin{equation}\label{bG}
\begin{xy}
(-5,17)*+{\bG_q = \underbrace{T\G \times_\G \dots \times_\G T\G}_{q-times}\times_\G \, \tar^*C^*} ="A"; (40,20)*+{\G} = "B";
(-5,-3)*+{\bM_q = \underbrace{TM \times_M \dots \times_M TM}_{q-times}\times_M C^*} = "C"; (40,0)*+{M} = "D";
{\ar@{->} (22,20)*{}; "B"}
{\ar@{->} (24,0)*{}; "D"}
{\ar@{->} (-5,11)*{}; (-5,3)*{}}
{\ar@{->} (-3,11)*{}; (-3,3)*{}}
{\ar@{->} (37,17)*{}; (37,3)*{}}
{\ar@{->} (34,17)*{}; (34,3)*{}}
\end{xy}
\end{equation}
where the structure maps are defined\footnote{There is a more general fact playing a role here: Whitney sums of VB-groupoids yield VB-groupoids (see \cite{BC}).} componentwise and $\tar^*C^* \toto C^*$ is the action groupoid corresponding to the right action of $\G$ (see Example \ref{ex:repG}) on $C^*$ obtained by taking adjoints. We shall frequently omit the subscript $q$ when no confusion arises. The $q$-fold fiber products on \eqref{bG} shall also be denoted as $\times^q_\G T\G \times_{\G} \tar^*C^*$ and $\times^q_M TM \times_M C^*$.

\begin{lemma}The space of $p$-composable arrows $B_p\bG$ is isomorphic as a vector bundle over $B_p\G$ to the $q$-fold fiber product $T B_p\G \times_{B_p\G} \dots \times_{B_p\G} TB_p\G \times_{B_p\G} \tar^*C^*$. More concisely,
\begin{equation}\label{re_arrag}
B_p\bG = B_p(\times_{\G}^q T\G \times_{\G} \tar^*C^*) \cong (\times^q_{B_p\G} T B_p\G) \times_{B_p\G} \tar^*C^*.
\end{equation}
\end{lemma}

The proof consists in simply defining the isomorphism:
$$
B_p\bG \ni (\underline{U}^{(1)}, \dots, \underline{U}^{(p)}) \mapsto ((U_1^{(1)}, \dots, U_1^{(p)}), \dots, (U_q^{(1)}, \dots, U_q^{(p)}), (g_1, \dots, g_p, \xi_1)),
$$
 where each $\underline{U}^{(i)}=(U_1^{(i)}, \dots, U_q^{(i)}, (g_i,\xi_i)) \in \bG$.

One important consequence of the isomorphism  \eqref{re_arrag} is that the space of differential forms $\Omega^q(B_p\G, \tar^*C)$ can be identified with a subspace of $C^{\infty}(B_p\bG)$, which we shall denote by $C^{\infty}_\ext(B_p\bG)$.
It is the image of the map \ref{diff_iso}:
\begin{equation}
\mathfrak{F}: \Omega^q(B_p\G, \tar^*C) \to C^{\infty}((\times^q_{B_p\G} T B_p\G) \times_{B_p\G} \tar^*C^*) \cong C^\infty(B_p\bG).
\end{equation}

In order to characterize $C^{\infty}_\ext(B_p\bG)$ more explicitly, note that, given a permutation $\sigma \in S_q$, the permutation map $\sigma_{\G}: \times^q_{\G} T\G \times_{\G} \tar^*C \to \times^q_{\G} T\G \times_{\G} \tar^*C$ is a groupoid morphism and, under the isomorphism \eqref{re_arrag},
\begin{equation}\label{perm_relation}
B_p \sigma_\G \cong \sigma_{B_p\G},
\end{equation}
for the corresponding permutation map $\sigma_{B_p\G}: \times^q_{B_p\G} TB_p\G \times_{B_p\G} \tar^*C \to \times^q_{B_p\G} TB_p\G \times_{\G} \tar^*C$. Similarly, the zero maps $0^{\G}_i: \bG_{q-1} \to \bG_q$, $i=1, \dots, q$ and $0^\G_{q+1}: \times_\G^q T\G \to \bG_q$ are groupoid morphisms and 
\begin{equation}\label{zero_relation}
B_p 0_i^{\G}  \cong 0^{B_p\G}_i, \,\, \forall \, i=1, \dots, q+1.\\
\end{equation}
Hence, 
$$
C^{\infty}_\ext(B_p\bG) =  \{f \in C^\infty_{(q+1)\mbox{-}\hom}(B_p\bG) \,\, | \,\, (B_p\sigma_\G)^*f = sgn(\sigma) f,\,\, (B_p0_i^\G)^*f = 0, \, \forall \sigma \in S_q, \, i=1, \dots, q+1\}.
$$
%
Note that the projection \eqref{proj:ext} gives here, under the isomorphism \eqref{re_arrag}, a projection $P_{\ext,\G}: C^{\infty}(B_p\bG) \to C^{\infty}_\ext(B_p\bG)$. 

\begin{proposition}
The projection $P_{\ext, \G}$ satisfies
$$
P_{\ext,\G} \circ \delta = \delta \circ P_{\ext, \G}
$$
In particular, $C^{\infty}_\ext(B_\bullet\bG)$ is a subcomplex.
\end{proposition}
\comment{Commutation of the projection and differential.}
\begin{proof}
The result follows directly from \eqref{perm_relation}, \eqref{zero_relation} and the fact that
$$
(B_{p+1}\phi)^*\delta f = \delta \, (B_{p}\phi)^*f,
$$
for an arbitrary groupoid morphism $\phi: \H_1 \to \H_2$ and $f \in C^p(\H_2)$.
\end{proof}

In the following, we shall denote by $C^\bullet_\ext(\bG)$ and $H^\bullet_\ext(\bG)$ the complex $(C^{\infty}_\ext(B_{\bullet}\bG), \delta)$ and its cohomology, respectively.
\begin{proposition}\label{commutation}
The map $\frakf: \Omega^q(B_p\G, \tar^*C) \to C^{p}_\ext(\bG)$ is a dg-module isomorphism.
\end{proposition}
\begin{proof}
Let $\partial_i: B_{p+1}\G \to B_{p} \G$, $\bpartial_i: B_{p+1}\bG \to B_p\bG$, $i=0,\dots, p+1$, be the face maps and $s_j: B_{p-1}\G \to B_{p}\G$, $\bs_j: B_{p-1}\bG \to B_p\bG$, $j=0, \dots, p-1$, be the degenaracy maps for $\G$ and $\bG$, respectively. The result follows from the fact that
\begin{align*}
\bpartial^*_0 \frakf_\omega & = \frakf_{g_1 \cdot \partial_0^*\omega}, \,\,\, \bpartial^*_i \frakf_\omega = \frakf_{\partial_i^*\omega},\\
\bs^*_j \frakf_\omega & = \frakf_{s_j^*\omega},\,\,\, \forall \,\omega \in \Omega^{q}(B_p \G, \tar^*C),
\end{align*}
when restricted to the fiber over $(g_1,\dots, g_p) \in B_p\G$.
\end{proof}


\paragraph{The Weil complex.}
The Lie algebroid $\bA_q \to \bM$ of the Lie groupoid \eqref{bG} $\bG_q \toto \bM_q$ is the $q$-fold fiber-product\footnote{As with VB-groupoids, Whitney sums of VB-algebroids yield VB-algebroids. Moreover, Whitney sums are preserved by the Lie functor (see \cite{BC}).}  $\times_\g^q T\g \times_\g \pi^*C^* \to \times_M^q TM \times_M C^*$, where $\pi: \g \to M$ denotes the projection map of the Lie algebroid of $\G$. 

\begin{definition}
Let $\alpha \in \Gamma(\bM, \Lambda^\bullet \bA^*)$. We say that $\alpha$ is skew-symmetric with respect to $\bA \to \g$ if
\begin{equation}
\sigma_\g^*\alpha  = sgn(\sigma) \alpha, \,\, \forall \, \sigma \in S_q
\end{equation}
where $\sigma_\g: \bA_q \to \bA_q$ permutes the $q$-coordinates on $\times_\g^q T\g $ according to $\sigma$. Similarly, $\alpha$ is multi-linear with respect to $\bA \to \g$ if
\begin{align}
h_\l^{\g \, *} \alpha & = \l^{q+1} \alpha\\
(0_{i}^\g)^* \alpha  & = 0, \,\, \forall \, i = 1, \dots, q+1
\end{align}
where $\h_\l^\g: \bA_q \to \bA_q$ is the homogeneous structure of the vector bundle $\bA_q \to \g$, $0^\g_i: \bA_{q-1} \to \bA_q$ and $0^\g_{q+1}: \times_{\g}^q T\g \to \bA_q$, $i=1,\dots, q$, are the zero maps.
\end{definition}

Let $\Gamma_{\ext}(\bM, \Lambda^{p} \bA_q^{*})$
 be the subspace of $\Gamma(\bM, \Lambda^{p} \bA_q^{*})$ of skew-symmetric multi-linear forms with respect to $\bA \to \g$. In particular, $\Gamma_{\ext}(\bM, \Lambda^{p} \bA_q^{*}) \subset \Gamma_{(q+1)\mbox{-}\hom}(\bM, \Lambda^{p}\bA_q)$. In the following, we shall frequently omit the reference to $q$ on the Lie algebroid $\bA_q$. There exists a projection $P_{\ext, \g}: \Gamma(\bM, \Lambda^p \bA) \to \Gamma_{\ext}(\bM, \Lambda^{p} \bA^{*})$ obtained exactly as \eqref{proj:ext} composing the projection $P^{\g,p}_{(q+1)\mbox{-}\hom}$ \eqref{proj:hom2} with the ones constructed from the zero maps $0^\g_i$ and permutations $\sigma^\g$ exactly as in \eqref{proj:simple} and \eqref{proj:skew}, respectively. 

\begin{proposition}
The projection $P_{\ext,\g}$ satisfies
$$
P_{\ext,\g} \circ d = d \circ P_{\ext,\g}.
$$
In particular, $\Gamma_{\ext}(\bM, \Lambda^{\bullet} \bA^{*})$ is a subcomplex of $\CE^\bullet(\bA)$.
\end{proposition}
\comment{Commutation of the projection and the differential.}
\begin{proof}
The result follows from the fact that the maps $h_\l^\g, 0_i^\g, \sigma_\g$ are all Lie algebroid morphisms. In fact, 
\begin{equation}\label{Lie_relations}
h_\l^\g = Lie(\h_\l^\G), \, 0_i^\g = Lie(0_i^\G), \, \sigma_\g = Lie(\sigma_\G),
\end{equation}
for the corresponding maps $\h_\l^\G, 0_i^\G, \sigma_\G$ on the Lie groupoid $\bG$.
\end{proof}

In the following, we shall denote by $\CE_\ext^\bullet(\bA)$ and $H^\bullet_\ext(\bA)$ the complex $(\Gamma_{\ext}(\bM, \Lambda^{\bullet} \bA^{*}), d)$ and its cohomology, respectively. Note that $\CE_{\ext}^\bullet(\bA)$ is a right $dg$-module for $\Gamma(\Lambda^\bullet \g) \cong \Gamma_{0\mbox{-}\hom}(\bM, \Lambda^\bullet\bA^*)$ by considering the wedge product.

\begin{proposition}\label{prop:evaluation}
There exists a right  $\Gamma(\Lambda^\bullet\g^*)$-module isomorphism $ev: \CE_{\ext}^\bullet(\bA) \to W^{\bullet,q}(\g, C)$ satisfying
$$
ev \circ d = \mathrm{d}_W \circ ev.
$$
\end{proposition}

We refer to the Appendix (see Proposition \ref{prop:evaluation2}) for a proof. It is important to note that Proposition \ref{prop:evaluation} implies that $W^{\bullet,q}(\g,C)$ is a right dg-module for $\Gamma(\wedge^\bullet\g^*)$ as stated in Proposition \ref{Weil:dg}. It is also worth noting that $ev$ is a map defined similarly to \eqref{evmap} (i.e. it evaluates an element $\alpha \in \Gamma_{\ext}(\bM, \Lambda^{p} \bA^{*})$ on a set of generators of $\Gamma(\bM,\bA)$ to give the sequence $(c_0, c_1, \dots) \in W^{p,q}(\g, C)$).  

\begin{remark} 
An alternative characterization of $\Gamma_{\ext}(\bM, \Lambda^{p} \bA^{*})$ can be given by seeing vector bundles as Lie groupoids (with multiplication given by addition on the fibers). Denote $\bA^{(p)} = \times_\bM^p \bA$ and $\g^{(p)} = \times_M^p\g$.  One has $\bA^{(p)} = B_p \bA$ and  $\g^{(p)}=B_p\g$. . In particular, the isomorphism \eqref{re_arrag} implies that
\begin{equation}\label{Ap}
\bA^{(p)} \cong \times_{\g^{(p)}}^q T\g^{(p)} \times_{\g^{(p)}} \pi^*C^*,
\end{equation}
as vector bundles over $\g^{(p)}$, where $\pi: \g^{(p)} \to M$ is defined (following the previous convention for $t: \G^{(p)} \to M$) as $\pi(u_1, \dots, u_p) = \pi(u_1)$. Hence, $\Omega^q(\g^{(p)}, \pi^*C)$, the space of differential forms on $\g^{(p)}$ with values on $C$, can be embedded as a subspace of $\C(\bA^{(p)})$ via \eqref{diff_iso}. Similarly, $\Gamma(\bM, \Lambda^p \bA)$ can also be embedded as a subspace of $\C(\bA^{(p)})$. One can now check that
$$
\Gamma_{\ext}(\bM, \Lambda^{p} \bA^{*}) = \Gamma(\bM, \Lambda^p \bA) \cap \Omega^q(\g^{(p)}, \pi^*C).
$$
We would like to mention that in the case where $C=\R$, with the trivial representation, Li-Bland and Meinrenken \cite{L-BM} gave a similar characterization of the Weil algebra as a subspace of differential forms on $\g$. In this context, the case $p=1$ was already studied by Bursztyn-Cabrera-Ortiz in \cite{BC, BCO}.
\end{remark}

\subsection{Proof of the Van Est Theorem for differential forms with coefficients}
Let $\VE: (C^{\bullet}(\bG), \delta_{\bG}) \to (\Gamma(\bM, \Lambda^{\bullet} \bA^*), d)$ be the Van Est map \eqref{vanest} for the groupoid $\bG \toto \bM$. 
\begin{proposition}
$$
\VE \circ P_{\ext}^{\G,p}  = P_{\ext}^{\g,p} \circ \VE.
$$ 
In particular, $\VE(C^{\infty}_{\ext}(B_p\bG)) \subset \Gamma_{\ext}(\bM, \Lambda^{p} \bA^{*})$.
\end{proposition}

\begin{proof}
From Proposition \eqref{hom_vanest}, one already has that $\VE$ satisfies $\VE\circ P_{q+1\mbox{-}\hom}^{\G,p} = P_{q+1\mbox{-}\hom}^{\g,p} \circ \VE$. So, it remains to show that $\VE$ commutes with the projections associated to the skew-symmetry and the simplicity properties. But this follows from Lemma \ref{commut_vannest} together with the relations \eqref{perm_relation}, \eqref{zero_relation} and \eqref{Lie_relations}.
\end{proof}

Let $\VE_{\ext}: C^{\infty}_\ext(B_p\bG) \to \Gamma_{\ext}(\bM, \Lambda^{p} \bA^{*})$ be the restriction of the Van Est map. 

\begin{lemma}\label{tech_lemma}
The following diagram 
\begin{equation}
\begin{xy}
(0,35)*+{\Omega^q(B_p\G, \tar^*C)}="A"; (40,35)*+{C_\ext^{\infty}(B_p\bG)}="B"; 
(0,20)*+{W^{p,q}(\g,C)}="C"; (40,20)*++{{\Gamma_{\ext}(\bM, \Lambda^{p} \bA^{*})}}="D";
{\ar@{->}^{\f} "A"; "B"}
{\ar@{->}^{\VE_{\ext}}"B";"D"}
{\ar@{->}^{\mathrm{ev}}"D";"C"}
{\ar@{->}_{\VE_{\Omega}}"A";"C"}
\end{xy}
\end{equation}
commutes.
\end{lemma}
 
The proof of Lemma \ref{tech_lemma} consists of  a direct but technical verification that we postpone until Appendix \ref{weil}. 
Finally, we are ready to prove our Theorem \ref{diff_vanest}. 

\medskip

\begin{proof}(\textbf{of Theorem \ref{diff_vanest}})
As $\mathrm{ev}$ and $\f$ are dg-module isomorphisms, it remains to show that $\VE_{\ext}$ induces isomorphisms on the cohomology $H^{p}(C^{\infty}_\ext(B_\bullet\bG)) \to H^{p}(\Gamma_{\ext}(\bM, \Lambda^{\bullet} \bA^{*}))$ for $p \leq p_0$ and a mononomorphism for $p=p_0+1$. Since the ordinary Van Est map $\VE_\bG$ for $\bG$ satisfyies the above, the theorem then follows from the Homological Lemma by means of the underlying projections exactly as in the proof of Theorem \ref{main}.
\end{proof}

\comment{Remark about the other module structure:}
\begin{remark}
The space $\Omega^\bullet(B_\bullet\G, \tar^*C)$ is a bigraded right module for the bigraded algebra $\Omega^\bullet(B_{\bullet}\G)$ with the cup product \cite{dupont}. The multiplication is given by
$$
\omega \cup \eta = (-1)^{qp'}\pr^*\omega \wedge \pr'^*\eta,\,\, \omega \in \Omega^q(B_p\G, \tar^*C), \, \eta \in \Omega^{q'}(B_{p'}\G),
$$  
where $\pr: B_{p+p'}\G \to B_{p}\G$ (resp. $\pr':B_{p+p'}\G \to B_{p'}\G$) is the projection onto the first $p$ arrows (resp. last $p'$ arrows). It is interesting to note that such module structure can also be described within the VB-groupoid context. Indeed, by considering the projections $\widetilde{\pr}: \bG^{q+q'} \to \bG^{q}$  and $\widetilde{\pr'}: \bG^{q+q'} \to \times^{q'}_\G T\G$, one can check that $\f_{\omega \cup \eta} \in C^{\infty}(B_{p+p'}\bG^{q+q'})$ can be obtained from $(B_{p}\widetilde{\pr})^*\f_\omega \in C^{\infty}(B_{p}\bG^{q+q'})$ and  $(B_{p'}\widetilde{\pr'})^*\f_{\eta} \in  C^{\infty}(B_{p'}\bG^{q+q'})$ by skew-symmetrizing their cup product
$$
(B_{p}\widetilde{\pr})^*\f_\omega \star (B_{p'}\widetilde{\pr'})^*\f_{\eta} \in C^{\infty}(B_{p+p'}\bG^{q+q'}).
$$
Similarly, one can define a bigraded module structure on $W^{\bullet,\bullet}(\g, C)$ for the Weil algebra $W^{\bullet,\bullet}(\g)$ \cite{AC} using the wedge product for their models as subcomplexes of the Chevalley-Eilenberg complexes. These bigraded module structures should be useful for studying ``multiplicative linear flat'' connections on $C$.  
\end{remark}

\appendix


\section{Formulas for the evaluation map}\label{app:formulas}
In this Appendix we turn to the proof of Lemma \ref{tech_lemma} relating the formula for $\VE_\Omega$ with the standard Van Est map for $\bG$ and $\bA$. In the process, we also give a detailed description (c.f. eq. \eqref{eq:evdef} below) of the map $ev:\Gamma_{\ext}(\bM, \Lambda^{p} \bA_q^{*}) \to W^{p,q}(\g,C)$ making use of special sections of $\bA_q$.

\subsection{Special sections}
Let $TB \rightarrow B$ be the tangent bundle of $B$. Given a vector field $X \in \mathfrak{X}(B)$, let $X^T, X^\vl \in \mathfrak{X}(TB)$ be its \textit{tangent} and \textit{vertical lift}\footnote{The flow at time $\epsilon$ of $X^T$ (resp. $X^\vl$) is the derivative of the flow at time $\epsilon$ of $X$ (resp. translation by $\epsilon X$)}. respectively. Define $X^{T,q}, \, X^{\vl,q}_{(j)}$, $j=1, \dots, q$ vector fields on the manifold $\times_B^q TB$ as follows
\begin{align}
X^{T,q}(v_1, \dots, v_q) & = (X^T(v_1), \dots, X^T(v_q))\\
X^{\vl,q}_{(j)}(v_1, \dots, v_q) & = (0_{v_1}, \dots, X^\vl(v_j), \dots, 0_{v_q}).
\end{align}

Let now $\G \toto M$ be a Lie groupoid with Lie algebroid $\pi: \g \to M$. For a representation $C \to M$ of $\G$, consider the Lie groupoid \eqref{bG}, $\times_\G^q T\G \times_\G \times \tar^*C^* \toto \times_M^q TM \times_M C^*$, with corresponding Lie algebroid $\times_\g^q T\g \times_\g \pi^*C^*$. For a section $u: M \to \g$, let $Tu: TM \to T\g$ be its derivative and $\chi_u : C^* \to \pi^*C^*=C^*\times_M \g$ the section defined by \eqref{linear_C}. The expressions
\begin{align*}
\T u(x_1, \dots, x_q,\xi) & = (Tu(x_1), \dots, Tu(x_q), \chi_u(\xi))\\
\Z_iu(x_1, \dots, x_q, \xi) & = (T0(x_1), \dots , T0(x_i) + \left.\frac{d}{d\epsilon}\right|_{\epsilon=0} \hspace{-15pt}(\epsilon u(m))\,, \dots, T0(x_q), 0_\xi),
\end{align*}
for $i=1, \dots, q$, $x_1, \dots, x_q \in T_mM$, $\xi  \in C^*_m$ and $m \in M$, define sections of the Lie algebroid $\bA=\times^q_\g T\g \times_\g \pi^*C^* \to \bM=\times^q_M TM \times_M C^*$. It is known that $\T u$ and $\Z_iu$, $i=1, \dots, q$, generate $\Gamma(\bM, \bA)$ as a $C^{\infty}(\bM)$ module \footnote{This follows from a general result regarding \textit{core} and \textit{linear sections} of double vector bundles (see \cite{Mac-doubles})}.

\begin{lemma}\label{lem:BpT}
As vector fields on $B_p( \times_\G^q T\G \times_\G \tar^*C^*) \cong \times_{B_p\G}^q TB_p\G \times_{B_p\G} \tar^*C^*$, the following identities hold
\begin{align}
\label{lift1} B_p(\T u) &= ((B_pu)^{T,q}, X_u) \\
\label{lift2} B_p(\Z_iv) &= ((B_pv)^{\vl,q}_{(i)}, 0),
\end{align}
where $X_u \in \mathfrak{X}(\tar^*C^*)$ is the vector field whose time-$\epsilon$ flow is given by
$$
(g_1,\dots, g_p, \xi) \mapsto (\psi_u^\epsilon(g_1,\dots, g_p), \Delta^*_{\phi_u^\epsilon(\tar(g_1))^{-1}}(\xi)).
$$
In particular, for $\omega \in \Omega^q(B_p\G, \tar^*C)$,
\begin{align}
\label{Lie} \bs_0^* \Lie_{B_p (\T u)}\f_{\omega} & = \f_{R_u \omega}\\
\label{contraction} \bs_0^* \Lie_{B_p (\Z_i u)}\f_{\omega} & = (-1)^{i-1} \f_{J_u \omega} \circ \pr_{(i)}^{p-1,q},
\end{align}
where $R_u$ and $J_u$ were defined in eq. \eqref{eq:RJ}, $\pr_{(i)}^{\cdot,q}: \times_{B_\cdot\G}^q TB_\cdot\G\times_{B_\cdot\G} \tar^*C^*\to \times_{B_\cdot \G}^{q-1} TB_\cdot\G\times_{B_\cdot\G} \tar^*C^*$ is the projection which forgets the $i$-th component and $\bs_0$ is the first degeneracy map for $\bG$.
\end{lemma}

\begin{proof}
For $u \in \Gamma(\g)$, consider the sections $\Z u, Tu$ of $T\g \to TM$, where
$$
\Z u (x) =T0(x) + \left.\frac{d}{d\epsilon}\right|_{\epsilon=0}(\epsilon u(m)), \,\, x \in T_mM.
$$
One has that
$\overrightarrow{T u} = \overrightarrow{u}^{T}$,
$\overrightarrow{\Z u} = \overrightarrow{u}^{\vl}$,
as vector fields on $T\G \toto TM$ (see \cite{Mac-Xu}). Also, the flow of the right invariant vector field $\overrightarrow{\chi_u} \in \mathfrak{X}(\tar^*C^*)$ is given by
$$
(g,\xi) \mapsto (\phi_u^\epsilon(g), \phi_u^\epsilon(\tar(g))^{-1}\cdot \xi).
$$
The identities \eqref{lift1}, \eqref{lift2} now follow from analysing the flows together with the rearragement isomorphism \eqref{re_arrag}. Hence, for $\omega \in \Omega^q(B_p\G, \tar^*C)$, 
\begin{align*}
(\Lie_{B_p(\T u)} \f_\omega)|_{(\overline{U}_1, \dots, \overline{U}_q, (g_1, \dots, g_p, \xi))} & = \left.\frac{d}{d\epsilon}\right|_{\epsilon=0}  \f_\omega(T\psi_u^\epsilon(\overline{U}_1), \dots, T\psi_u^\epsilon(\overline{U}_q), (\psi_u^\epsilon(g_1,\dots, g_p), \phi_u^\epsilon(\tar(g_1))^{-1}\cdot \xi))\\
& = \<\xi, \phi(\tar(g_1))^{-1}\cdot (\psi_u^\epsilon)^*\omega(\overline{U}_1, \dots, \overline{U}_q)\>
\end{align*}
Now, \eqref{Lie} follows from the commutation relations on Proposition \ref{commutation}. The identity \eqref{contraction} follows similarly.
\end{proof}

\subsection{The evaluation map}\label{weil}
We shall now describe the chain isomorphism $ev: \Gamma_{\ext}(\bM, \Lambda^{p} \bA^{*}) \to W^{p,q}(\g,C)$. First, for $\alpha \in \Gamma_{\ext}(\bM, \Lambda^{p} \bA^{*}) \subset \Gamma(\bM, \Lambda^p\bA^*)$, define 
$$
\tilde{c}_k(\alpha): (\times^{p-k} \Gamma(\g))\times ( \times^k \Gamma(\g) )\rightarrow C^{\infty}(\times_M^{q} TM \times_M C^*)
$$
as:
$$
\tilde{c}_k(\alpha)(u_1, \dots, u_{p-k}|v_1, \dots, v_k)= \alpha(\Z_{1} v_1, \dots, \Z_{k} v_k, \T u_1, \dots, \T u_{p-k} \vphantom{\frac{1}{1}}).
$$

\begin{lemma}
There exists a map $c_k(\alpha): \times^p \Gamma(\g) \to \Omega^{q-k}(M, C)$ such that 
\begin{equation}\label{c_alpha}
\tilde{c}_k(\alpha) =\f_{c_k(\alpha)} \circ  \pr_{[1,k]}, \,\,\, \forall \,\alpha \in \Gamma_\ext(\bM,\Lambda^p\bA^*_q),
\end{equation}
where $\pr_{[1,k]}: \times_M^q TM\times_M C^* \to \times_M^{q-k} TM\times_M C^*$ is the projection which forgets the first $k$ entries.
\end{lemma}

\begin{proof}
The multilinearity with respect to the both vector bundle structures, $\bA \to \bM$ and $\bA \to \g$, implies that
$
\tilde{c}_k(\alpha)(u_1, \dots, u_{p-k}|v_1, \dots, v_k) = F_\alpha \circ \pr_{[1,k]},
$
where $F_\alpha \in C^{\infty}(\times^{q-k}TM\times_M C^*)$ is given by
$$
F_\alpha(y_1, \dots, y_{q-k}, \xi) = \tilde{c}_k(\alpha)(u_1, \dots, u_{p-k}|v_1, \dots, v_k)(\underbrace{0_m, \dots, 0_m}_{k-times}, y_{1}, \dots, y_{q-k}, \xi).
$$
We now have to check that $F_\alpha \in C_{\ext}^\infty(\times^{q-k}TM\times_M C^*)$, i.e. $F$ is $q-k+1$-homogeneous, simple and skew-symmetric.
The homogeneity of $F_\alpha$ follows from the homogeneity of $\alpha$ together with the linearity of the sections $\T u$ and the properties of the section $\Z_jv$: 
$$
\Z_j(v)|_{(0_m, \dots, 0_m, \l y_1, \dots, \l y_{q-k},\l \xi)} = h_\l^\g\left(\Z_j(\frac{1}{\l} v)|_{(0_m,\dots, 0_m, y_1, \dots, y_{q-k}, \xi)}\right), \Z_j(\l v) = \l \cdot \Z_j(v), \,\, \l > 0,
$$
where $\cdot$ stands for the multiplication for $\bA \to \bM$.
The simplicity of $F_\alpha$ follows from the identity
\begin{align*}
(F_\alpha \circ 0_i) \circ \pr_{[1,k]} = ((0_{k+i}^\g)^*\alpha)( \Z_1v_1, \dots, \Z_kv_k, \T u_1, \dots, \T u_{p-k}) = 0,
\end{align*} 
for $i=1, \dots, q-k+1$.
Finally, let $\sigma \in S_{q-k} \subset S_q$, seen as the subgroup acting as the identity on $\{1,\dots, k\}$. One can check that
\begin{align*}
(F_\alpha  \circ \sigma_M)\circ \pr_{[1,k]} & = (\sigma_\g^* \alpha)(\Z_1v_1, \dots, \Z_kv_k, \T u_1, \dots, \T u_{p-k})\\
& = sgn(\sigma) \,\alpha(\Z_1v_1, \dots, \Z_kv_k, \T u_1, \dots, \T u_{p-k})\\
& = sgn (\sigma) F_\alpha \circ \pr_{[1,k]}.
\end{align*}
This shows that $F_\alpha \in C^{\infty}_\ext(\times^{q-k} TM\times_M C^*)$ and, therefore, there exists $c_k(\alpha): \times^{p} \Gamma(\g) \to \Omega^{q-k}(M, C)$ such that $F_\alpha= \f_{c_k(\alpha)(u_1, \dots, u_{p-k}| v_1, \dots, v_k)}$.
\end{proof}

Our aim is to prove that 
\begin{equation} \label{eq:evdef}
 ev(\alpha)= (c_0(\alpha), c_1(\alpha), \dots)
\end{equation}
defines a map from $\Gamma_\ext(\bM, \Lambda^p\bA^*)$ into $W^{p,q}(\g, C)$. First note that the sequence $(c_0(\alpha), c_1(\alpha), \dots)$ completely determines $\alpha \in \Gamma_{\ext}(\bM, \Lambda^{p} \bA^{*})$. Indeed, as $\Gamma(\bM, \bA)$ is generated as $C^{\infty}(\bM)$-module by sections of the type $\T u$, $\Z_i v$, any element of $\Gamma(\bM, \Lambda^p\bA^*)$ is determined by its values on these sections. Now, one can check that, for $\alpha \in \Gamma_{\ext}(\bM, \Lambda^{p} \bA^{*})$,
\begin{equation}\label{eq1}
i_{\Z_j v}i_{\Z_j w} \, \alpha =  0, \,\, \text{for} \, j=1, \dots, q
\end{equation}
and, for a permutation $\sigma \in S_q$,
\begin{equation}\label{eq2}
 \alpha(\Z_{\sigma(1)} v_1, \dots, \Z_{\sigma(k)} v_k, \T u_1, \dots, \T u_{p-k}) = sgn(\sigma) \,\sigma_M^*(\alpha(\Z_1 v_1, \dots, \Z_k v_k, \T u_1, \dots, \T u_{p-k})).
\end{equation}
Hence, to recover $\alpha$ from its values on the sections $\T u, \Z_i v$, it suffices to know the values of $\alpha$ encoded on the sequence $(c_0(\alpha), c_1(\alpha), \dots)$. The next result gives the desired proof of Proposition \ref{prop:evaluation}.

\begin{proposition}\label{prop:evaluation2}
Given $\alpha \in \Gamma_{\ext}(\bM, \Lambda^{p} \bA^{*})$, one has that: 
\begin{enumerate}
\item $c_k(\alpha)$ is skew-simmetric on the $u$'s entries.
\item $c_k(\alpha)$ is symmetric on the $v$'s entries.
\item Given $f \in C^{\infty}(M)$,
\begin{align*}
c_k(\alpha)(u_1, \dots,  u_{p-k}| v_1, \dots, f v_k) & = f c_k(\alpha)(u_1, \dots, u_{p-k}| v_1, \dots, v_k)\\
c_k(\alpha)(f u_1, \dots, u_{p-k}| v_1, \dots, v_k) & = f c_k(\alpha)(u_1, \dots, u_{p-k}| v_1, \dots, v_k)\\
& \hspace{-50pt} +  df \wedge c_{k+1}(\alpha)(u_2,\dots, u_{p-k}| v_1, \dots, v_k, u_{1})
\end{align*}
\end{enumerate}
In particular, each $c_k$ can be seen as a $\R$-linear skew-symmetric map $c_k: \times^{p-k} \Gamma(\g) \to \Omega^{q-k}(M, S^{k} \g^* \otimes C)$. Moreover, the map $ev: \Gamma_{\ext}(\bM, \Lambda^{p} \bA^{*}) \to W^{p,q}(\g,C)$ defined by \eqref{eq:evdef} is a right $\Gamma(\Lambda^\bullet \g^*)$-module isomorphism satisfying
$$
ev \circ \mathrm{d}_\ext = \mathrm{d}_W \circ ev.
$$
\end{proposition}

\begin{proof}
\paragraph{1.} This follows directly from the skew-symmetry of $\alpha$ with respect to $\bA \to \bM$.
\paragraph{2.} Let $\sigma \in S_k \subset S_{q}$, seen as the subgroup acting as the identity on $\{k+1, \dots, q\}$. From \eqref{eq2} and the skew-symmetry of $\alpha$ with respect to $\bA \to \bM$,
\begin{align*}
\alpha(\Z_{1}v_{\sigma(1)}, \dots, \Z_{k}v_{\sigma(k)}, \T u_1, \dots, \T u_{p-k})& = sgn(\sigma)\, \alpha(\Z_{\sigma(1)}v_{\sigma(1)}, \dots, \Z_{\sigma(k)}v_{\sigma(k)}, \T u_1, \dots, \T u_{p-k})\\
& = (sgn(\sigma))^2 \, \alpha(\Z_{1}v_{1}, \dots, \Z_{k}v_{k}, \T u_1, \dots, \T u_{p-k}).
\end{align*}
In the second equality we have used the fact that $\alpha(\T u_1, \dots, \T u_{p-k}, \Z_{\sigma(1)}v_{\sigma(1)}, \dots, \Z_{\sigma(k)}v_{\sigma(k)}) \in C^{\infty}(\times^q TM \times_M C^*)$ does not depend on the first  $k$ coordinates.
\paragraph{3.}
One can check that 
\begin{align*}
\Z_i(fv) & = (f \circ \pi) \cdot \Z_iv\\
\T(fu) & = (f \circ \pi) \cdot \T u + \sum_{j=1}^q (\ell_{df} \circ \pr_{j}) \cdot \Z_j u
\end{align*}
where all the sums and scalar multiplications are with respect to $\pi: \bA \to \bM$, $\pr_{j}: \times_M^q TM \times_M C^*\to TM$ is the projection onto the $j$-th factor and $\ell_{df} \in C^{\infty}(TM)$ is the linear function corresponding to $df \in \Omega^1(M)$. To simplify notation, we identify $\Omega^{q-k}(M, C)$ with its image on $C^{\infty}(\times^{q-k}TM \times_M C^*)$ under $\f$ in the following. The first equation of (3) is now straightforward to check. As for the second, it follows from \eqref{eq1} and \eqref{eq2} that
\begin{align*}
c_k(\alpha)(fu_1, \dots, u_{p-k}| v_1, \dots, v_k) \circ \pr_{[1,k]}& =  (f \circ \pi) \,\alpha(\Z_1v_1, \dots, \Z_k v_k, \T u_1,\dots, \T u_{p-k}) \\ &  \hspace{-110pt} + \sum_{j=k+1}^{q} (\ell_{df} \circ \pr_{j}) \, \alpha(\Z_1 v_1, \dots, \Z_kv_k, \Z_j u_{1}, \T u_2, \dots, \T u_{p-k})\\
& = (f \circ \pi) c_k(\alpha)(u_1, \dots, u_{p-k}| v_1, \dots, v_k) \circ \pr_{[1,k]} \\
& \hspace{-113pt} + \underbrace{\sum_{j=k+1}^q  (-1)^{j-k-1} (\ell_{df} \circ \pr_{j})\,\alpha( \Z_1 v_1, \dots, \Z_kv_k, \Z_{k+1}u_{1}, \T u_2, \dots, \T u_{p-k}) \circ \sigma^j_M}_{(*)}
\end{align*}
where $\sigma^j \in S_q$ is the cycle $(j \,\, j-1 \,\, \dots k+2 \,\, k+1)$, for $k+1 \leq j \leq q$, which has sign equal to $(-1)^{j-k-1}$. It is now straightforward to check that $(*)$ equals  $df \wedge c_{k+1}(u_2, \dots, u_{p-k}| v_1, \dots, v_k, u_{1})\circ \pr_{[1,k]}$.

It remains to prove that $ev$ is a dg-module isomorphism. Let us first prove that $ev$ commutes with the multiplication. Let $\beta \in \Gamma(\Lambda^{p'}\g^*) \cong \Gamma_{0\mbox{-}\hom}(\bM, \Lambda^{p'} \bA^*)$
and consider $ev(\alpha \wedge \beta) = (c_0(\alpha\wedge\beta), c_1(\alpha\wedge\beta), \dots)$. By definition,
\begin{align*}
c_k(\alpha \wedge \beta)(u_1,\dots, u_{p+p'-k}|v_1, \dots, v_k)\circ \pr_{[1,k]} & = (\alpha \wedge \beta) (\Z_1v_1, \dots, \Z_kv_k, \T {u_1},\dots, \T{u_{p+p'-k}})\\
& \hspace{-150pt}= \hspace{-10pt} \sum_{\sigma \in S(p-k,p')} \hspace{-10pt} sgn(\sigma)\, \alpha(\Z_1v_1, \dots, \Z_kv_k, \T u_{\sigma(1)}, \dots, \T u_{\sigma(p-k)}) \,\beta(\T u_{\sigma(p-k+1)}, \dots \T u_{\sigma(p+p'-k)}).
\end{align*}
where $S(p-k,p')$ is the space of $(p-k,p')$-unshuffles and the last equality follows from the fact that the contraction of $\beta$ with any section of type $\Z_\cdot v_\cdot$ is zero. The result now follows easily.

Finally, to prove that $ev$ intertwines the differential, consider
$$
ev(d\alpha)= (c_0(d\alpha), \dots, c_k(d\alpha), \dots),
$$
where
\begin{align*}
c_k(d\alpha)(u_1, \dots, u_{p+1-k}|v_1, \dots, v_k) \circ \pr_{[1,k]} & = d\alpha(\Z_1v_1, \dots, \Z_kv_k, \T u_1, \dots, \T u_{p+1-k})\\
 & \hspace{-100pt}= \underbrace{\sum_{j=1}^{k} (-1)^{j+1} \Lie_{\rho(\Z_j v_j)} \alpha(\Z_1 v_1, \dots, \widehat{\Z_j v_j}, \dots)}_{(A)} + \, \underbrace{\sum_{i=1}^{p-k+1} (-1)^{i+k+1}\Lie_{\rho(\T u_i)} \alpha(\Z_1 v_1, \dots, \widehat{\T u_i}, \dots)}_{(B)} \\
& \hspace{-80pt}+ \underbrace{\sum_{1 \leq i < j \leq p+1-k} \!\!\!(-1)^{i+j} 
\alpha([\T u_i, \T u_j], \dots, \widehat{Tu_i},\dots, \widehat{Tu_j}, \dots)}_{(C)}\\
&\hspace{-80pt} + \underbrace{\sum_{i=1}^{p+1-k} \sum_{j=1}^k (-1)^{j +(k+i)} \alpha([\Z_jv_j, \T u_i], \dots, \widehat{\Z_jv_j} , \dots, \widehat{\T u_i}, \dots)}_{(D)}.\\
\end{align*}
Notice that there are no terms containing $[\Z_{j_1}v_{j_1}, \Z_{j_2}v_{j_2}]$  since these brackets are all zero. To study the remaining terms, we shall use some properties of the tangent Lie algebroid $T \g \to TM$ (see \cite{Mac-Xu}) and the action algebroid $C^*\times_M \g \to C^*$.

\begin{itemize}
\item[(A)] From \eqref{eq2},
\begin{align*}
\alpha(\Z_1v_1, \dots, \widehat{\Z_jv_j}, \dots, \Z_kv_k, \T u_1, \dots, \T u_{p+k-1}) & \\
& \hspace{-100pt} = (-1)^{k-j} \sigma_M^*(\f_{c_{k-1}(\alpha)(u_1, \dots, u_{p+1-k}|v_1, \dots, v_{j-1}, v_{j+1}, \dots, v_k)} \circ \pr_{[1,k-1]})
\end{align*}
where $\sigma = (j \,\, k) (j \,\, k-1) \dots (j \, j+1) \in S_q$.
Now,
$$
\pr_{[1,k-1]} \circ \sigma_M(x_1,\dots, x_q, \xi) = (x_j, x_{k+1}, \dots, x_q, \xi) \,\, \text{ and } \,\,\rho(\Z_jv_j) = (\rho(v_j)^{\vl,q}_{(j)}, 0)
$$ 
and
$$
\Lie_{(X^{\vl,q}_{(1)},0)} \f_\omega = \f_{i_X\omega} \circ \pr_{(1)}, \,\,\, \forall \,  X \in \mathfrak{X}(M), \, \omega \in \Omega^q(M, C)
$$
where $\pr_{(1)}: \times^q TM \times_M C^* \to \times^{q-1} TM \times_M C^*$ is the projection which forgets the first component. These facts imply that
\begin{align*}
(A) = (-1)^{k+1}\sum_{j=1}^k \left( i_{\rho(v_j)}c_{k-1}(\alpha)(u_1, \dots, u_{p+1-k}|v_1, \dots, v_{j-1}, v_{j+1}, \dots, v_k)\right) \circ \pr_{[1,k]}
\end{align*}
\item[(B)] The fact that $\rho(\T u_i) = (\rho(u_i)^T, \rho(\chi_u))$, where $(\cdot)^T$ stands for tangent lift implies that, for $\omega \in \Omega^q(M, C)$,  
$
\Lie_{\rho(\T u_i)} \f_{\omega} = \f_{u_i\cdot \omega},
$
where $u \cdot (\beta \otimes \mu) = \Lie_{\rho(u)} \beta \otimes \mu + \beta \otimes \nabla_u \mu$,
$\beta \in \Omega^q(M)$, $\mu \in \Gamma(C)$ and $\nabla$ is the $\g$-connection on $C$ defining the representation of $\g$ on $C$. Hence, 
$$
(B) = (-1)^k\sum_{i=1}^{p-k+1} (-1)^{i+1} \left(u_i \cdot c_k(u_1, \dots, \widehat{u_i}, \dots, u_{p+1-k}| v_1, \dots, v_k)\right) \circ \pr_{[1,k]}
$$
\item[\noindent(C,D)] 
\noindent From the identities $[\T u_i, \T u_j] = \T[u_i,u_j]$ and $[\T u_i, \Z_jv_j] = \Z_j[u_i,v_j]$, it is straightforward to check that 
\begin{align*}
(C) & = (-1)^k\sum_{1\leq i_1 < i_j \leq p-k+1}(-1)^{i_1+i_2} \left(c_k(\alpha)([u_{i_1}, u_{i_2}],u_1, \dots, \widehat{u_{i_1}}, \dots, \widehat{u_{i_2}}, \dots, u_{p-k+1}| v_1, \dots, v_k)\right)\circ \pr_{[1,k]}\\
(D) & = (-1)^{k} \sum_{i=1}^{p-k+1}\sum_{j=1}^k (-1)^{i} \left(c_k(\alpha)(u_1, \dots, \widehat{u_i}, \dots, u_{p+1-k} | v_1, \dots, [u_i,v_j], \dots, v_k)\right) \circ \pr_{[1,k]}
\end{align*}
\end{itemize}
Hence,
$$
(A)+(B)+(C)+(D)  = \mathrm{d}_W(c(\alpha))_k(u_1, \dots, u_{p-k+1}|v_1,\dots, v_k) \circ \pr_{[1,k]}
\Longrightarrow 
c(d\alpha)_k = \mathrm{d}_W(c(\alpha))_k,
$$
as we wanted.
\end{proof}

\subsection{The proof of Lemma \ref{tech_lemma}}
\begin{namedthm*}{Lemma \ref{tech_lemma} rephrased}
Let $\omega \in \Omega^q(B_p\G, \tar^*C)$ and consider $\VE_\Omega(\omega)=(c_0(\omega), c_1(\omega), \dots)$ as defined in \eqref{c_omega}. Also, let $\alpha = \VE_{\ext}(\f_\omega) \in \Gamma_{\ext}(\bM, \Lambda^p\bA^*)$ and consider $ev(\alpha) = (c_0(\alpha), c_1(\alpha), \dots)$ defined by \eqref{c_alpha}. Then,
$$
c_k(\omega) = c_k(\alpha), \forall \,\, k \geq 0.
$$
\end{namedthm*}

\begin{proof}
From \eqref{vanest},
\begin{align*}
\f_{c_k(\alpha)(u_1, \dots, u_{p-k}| v_1, \dots, v_k)}\circ \pr_{[1,k]} & = \VE_{\ext}(\f_{\omega})(\Z_1 v_1, \dots, \Z_k v_k, \T u_1, \dots, \T u_{p-k})\\
 & = \sum_{\sigma \in S_p} sgn(\sigma) R_{\chi_{\sigma(1)} }\dots R_{\chi_{\sigma(p)}} \f_\omega,
\end{align*}
where $\chi_i = \Z_i v_i$ (resp. $\T u_{i-k}$), if $i \in \{1, \dots, k\}$ (resp. if $i \in \{k+1, \dots, p\}$). The main ingredients of the proof are the identities from Lemma \ref{lem:BpT}:
\begin{align*}
R_{\T u_i} \f_{\omega} & = \mathbbm{s}_0^* \Lie_{B_p \T u_i}\f_\omega 
 = \f_{R_{u_i} \omega}\\
R_{\Z_iv_i} \f_{\omega} & = (-1)^{i-1}\f_{J_{v_i} \omega} \circ \pr_{(i)}^{p-1,q},
\end{align*}
where $\pr_{(i)}^{\cdot,q}: \times_{B_\cdot\G}^q TB_\cdot\G\times_{B_\cdot\G} \tar^*C^*\to \times_{B_\cdot \G}^{q-1} TB_\cdot\G\times_{B_\cdot\G} \tar^*C^*$ is the projection which forgets the $i$-th component. In the rest of the proof, the difficulty lies in the combinatorics needed to count the number of $(-1)$'s appearing due to the presence of the sections $\Z_i v_i$.

Let $0 \leq r \leq p$, $\, 1 \leq s \leq q$  and $1 \leq i,j \leq s$. For $\eta \in \Omega^{s-1}(B_r\G, \tar^*C)$, one can check that
$$
R_{\T u} (\frakf_{\eta} \circ \pr_{(i)}^{r,s})  = \frakf_{R_u \eta} \circ \pr_{(i)}^{r-1,s}
$$ and
\begin{align*}
R_{\Z_jv} (\frakf_{\eta} \circ \pr_{(i)}^{r,s}) & =
 \begin{cases}
(-1)^{j-1} \frakf_{J_v \eta} \circ \pr_{(j)}^{r-1,s-1} \circ \pr_{(i)}^{r-1,s} & \text{ if } i > j \vspace{5pt}\\
 -(-1)^{j-1} \frakf_{J_v \eta} \circ \pr_{(i)}^{r-1,s-1} \circ\pr_{(j)}^{r-1,s}& \text{ if } i < j\\
0 & \text{ if } i=j.
\end{cases}
\end{align*}
Let us now fix a permutation $\sigma \in S_q$. For $1 \leq l \leq k$,  let $j_l = \sigma^{-1}(l)$, for $l \geq 1$ and set $j_0=0$. Denote by $\tau$ the permutation of $\{0,1,\dots, k\}$ such that $j_{\tau(0)} < \dots < j_{\tau(k)}$.  One can now prove by induction that, for $j_{\tau(l)} \leq r < j_{\tau(l+1)}$,
\begin{equation}
R_{\chi_{\sigma(r+1)}} \dots R_{\chi_{\sigma(p)}} \f_\omega = \delta(k,l)  \left(\f_{D_{\sigma(r+1)} \dots D_{\sigma(p)} \omega} \circ \pr^{r, q-k+l}_{(i_1)} \circ \dots \circ \pr^{r, q}_{(i_{k-l)}} \right),
\end{equation}
where the $D_j$'s are the operators \eqref{def:D}, $\{i_1 < \dots < i_{k-l}\} = \{\tau(l+1), \dots, \tau(k)\}$ and 
$$
\delta(k,l) = (-1)^{k-l} (-1)^{\tau(l+1)+\dots + \tau(k)} (-1)^{N(\tau,l)}
$$ 
with
$$
N(\tau,l) = \#\{(i,j) \in \{l+1, \dots, k\} \times \{l+1, \dots, k\}\,\, | \,\, i < j \, \text{ and } \sigma^{-1}(i) > \sigma^{-1}(j)\}
$$
Note that, for $l=0$,
$$
\delta(k,0) = (-1)^{k}(-1)^{1+\dots + k} (-1)^{\epsilon(\sigma,k)} = (-1)^{\frac{k(k-1)}{2}}(-1)^{\epsilon(\sigma,k)}.
$$
In particular, when $r=0$, we have
\begin{align*}
\f_{c_k(\alpha)(u_1, \dots, u_{p-k}| v_1, \dots, v_k)}\circ \pr_{[1,k]}  & = \sum_{\sigma \in S_p} sgn(\sigma) R_{\chi_{\sigma(1)} }\dots R_{\chi_{\sigma(p)}} \f_\omega\\ 
& = (-1)^{\frac{k(k-1)}{2}} \sum_{\sigma \in S_p} sign(\sigma)(-1)^{\epsilon(\sigma,k)} \f_{D_{\sigma(1)} \dots D_{\sigma(p)} \omega} \circ \pr^{0, q-k}_{(1)} \circ \dots \circ \pr^{0, q}_{(k)} \\
& = \f_{c_k(\omega)(u_1, \dots, u_{p-k}| v_1, \dots, v_k)} \circ \pr_{[1,k]}
\end{align*}
This concludes the proof.

\end{proof}


\begin{thebibliography}{99}


\bibitem{AC}
Arias Abad, C., Crainic, M., The Weil algebra and the Van Est
isomorphism, {\em  Annales de l'Institut Fourier.} {\bf 61(3)}(2011), 927-970.

\bibitem{AC2}
Arias Abad, C., Crainic, M., Representations up to homotopy of Lie algebroids, {\em Journal f{\"u}r die reine und angewandte Mathematik.} {\bf 663} (2012), 91-126.

\bibitem{AC3}
Arias Abad, C., Crainic, M., Representations up to homotopy and Bott's spectral sequence for Lie groupoids,
{\em Advances in Mathematics.} {\bf 248} (2013), 416-452.

\bibitem{AS}
Arias Abad, C., Schatz, F., Deformations of Lie brackets and representations up to homotopy, {\em Indagationes Mathematicae} {\bf 22(1)} (2011), 27-54.

\bibitem{BCO}
Brahic, O., Cabrera, A., Ortiz, C., Obstructions to the integrability of VB-algebroids, {\em arXiv: 1403.1990 [math.DG]}.

\bibitem{BC}
Bursztyn, H., Cabrera, A., Multiplicative forms at the infinitesimal level, {\em Math. Ann.} {\bf 353(3)} (2012),  663-705.



\bibitem{BCdH} Bursztyn, H., Cabrera, A., del Hoyo, M., Vector bundles over Lie groupoids and algebroids, {\em arXiv:1410.5135 [math.DG]}.

%
%

%
%
%
\bibitem{bd}
Bursztyn, H., Drummond, T., Lie theory of multiplicative tensors, {\em Work in progress.}
%
%
%
%
%
%
\bibitem{Cr}
Crainic, M., Differentiable and algebroid cohomology, van Est isomorphisms, and characteristic classes, {\em Commentarii Mathematici Helvetici} {\bf 78(4)} (2003), 681-721.




\bibitem{CrMeSt}
Crainic, M., Mestre, J., Struchiner, I., Deformations of Lie groupoids {\em arXiv preprint arXiv:1510.02530} (2015).

\bibitem{CrMo} 
 Crainic, M., Moerdijk, I., Deformations of Lie brackets: cohomological aspects, {\em Journal
of the EMS}, {\bf 10} (2008), 1037-1059.

\bibitem{CSS}
Crainic, M., Salazar, M., Struchiner, I., Multiplicative forms and Spencer operators, {\em Mathematische Zeitschrift}  {\bf 279.(3-4)} (2014), 939-979.


%
%
%

\bibitem{dupont}
Dupont, J.L., Curvature and characteristic classes, {\em Lecture Notes in Mathematics}, {\bf 640} (1978).

\bibitem{vE1}
van Est, W.T., Group cohomology and Lie algebra cohomology in Lie groups I, II, 
{\em Proc. Kon. Ned. Akad.}, {\bf 56} (1953), 484-504.

\bibitem{vE2}
van Est, W.T., On the algebraic cohomology concepts in Lie groups I, II, {\em Proc. Kon. Ned. Akad.}, {\bf 58} (1955), 225-233, 286-294.

\bibitem{GM08} 
Gracia-Saz,A.,  Mehta, R., 
Lie algebroid structures on double vector bundles and representation theory of Lie algebroids, {\em  
Adv. Math.} {\bf 223.4} (2010), 1236-1275.

\bibitem{GM10} 
Gracia-Saz,A.,  Mehta, R., $\mathcal{VB}$-groupoids and representation theory of Lie groupoids, {\em preprint arXiv:1007.3658v4}

\bibitem{GR}
Grabowski, J., Rotkiewicz, M.; Higher vector bundles and
multi-graded symplectic manifolds; arxiv:math/0702772v2 [math.DG].
%
%
%
%
%
%
%
%
%
%
%
%
%
%
%
%
%
%
\bibitem{L-BM}
Li-Bland, D., Meinrenken, E. On the Van Est homomorphism for Lie groupoids, {\em arXiv preprint arXiv:1403.1191 (2014)}.

%
%
%
%
%
%
%
\bibitem{Mac-doubles}
Mackenzie, K., Ehresmann doubles and Drinfel'd doubles for Lie algebroids and Lie bialgebroids, {\em  J. Reine Angew. Math.} {\bf 658} (2011), 193-245. 
%
%
\bibitem{Mac-Xu}
Mackenzie, K., Xu, P., Lie bialgebroids and Poisson groupoids.  {\em Duke Math. J.}  {\bf 73}  (1994), 415--452
%
%

\bibitem{Mac-Xu3}
Mackenzie, K., Xu, P., Classical lifting processes and multiplicative vector fields. {\em The Quart. J. Math.} {\bf 49(1)} (1998): 59-85.

\bibitem{Mehta} 
Mehta, R., Q-groupoids and their cohomology, {\em Pacific J. Math.} 242 (2009), no. 2, 311-332.


%
%
%


\bibitem{Prad2}
Pradines, J., Remarque sur le groupoide cotangent de Weinstein-Dazord, {\em CR Acad. Sci. Paris Sér. I Math} {\bf 306(13)} (1988), 557-560.
NBR 6023	



%
%
%
%
%
%
%
\bibitem{XW}
Weinstein, A., Xu, P,  Extensions of symplectic groupoids and quantization, {\em J. reine angew. Math} {\bf 417} (1991),  159-189.

%


\end{thebibliography}
\end{document}